\documentclass[11pt]{amsart}
\usepackage{amssymb}
\usepackage{verbatim}
\usepackage{hyperref}
\usepackage{color}

\begin{document}

\newtheorem{theorem}{Theorem}    
\newtheorem{proposition}[theorem]{Proposition}
\newtheorem{conjecture}[theorem]{Conjecture}
\def\theconjecture{\unskip}
\newtheorem{corollary}[theorem]{Corollary}
\newtheorem{lemma}[theorem]{Lemma}
\newtheorem{sublemma}[theorem]{Sublemma}
\newtheorem{fact}[theorem]{Fact}
\newtheorem{observation}[theorem]{Observation}
\theoremstyle{definition}
\newtheorem{definition}{Definition}
\newtheorem{notation}[definition]{Notation}
\newtheorem{remark}[definition]{Remark}
\newtheorem{question}[definition]{Question}
\newtheorem{questions}[definition]{Questions}
\newtheorem{example}[definition]{Example}
\newtheorem{problem}[definition]{Problem}
\newtheorem{exercise}[definition]{Exercise}

\numberwithin{theorem}{section}
\numberwithin{definition}{section}
\numberwithin{equation}{section}

\def\reals{{\mathbb R}}
\def\torus{{\mathbb T}}
\def\heis{{\mathbb H}}
\def\integers{{\mathbb Z}}
\def\rationals{{\mathbb Q}}
\def\naturals{{\mathbb N}}
\def\complex{{\mathbb C}\/}
\def\distance{\operatorname{distance}\,}
\def\support{\operatorname{support}\,}
\def\dist{\operatorname{dist}\,}
\def\Span{\operatorname{span}\,}
\def\degree{\operatorname{degree}\,}
\def\kernel{\operatorname{kernel}\,}
\def\dim{\operatorname{dim}\,}
\def\codim{\operatorname{codim}}
\def\trace{\operatorname{trace\,}}
\def\Span{\operatorname{span}\,}
\def\dimension{\operatorname{dimension}\,}
\def\codimension{\operatorname{codimension}\,}
\def\nullspace{\scriptk}
\def\kernel{\operatorname{Ker}}
\def\ZZ{ {\mathbb Z} }
\def\p{\partial}
\def\rp{{ ^{-1} }}
\def\Re{\operatorname{Re\,} }
\def\Im{\operatorname{Im\,} }
\def\ov{\overline}
\def\eps{\varepsilon}
\def\lt{L^2}
\def\diver{\operatorname{div}}
\def\curl{\operatorname{curl}}
\def\etta{\eta}
\newcommand{\norm}[1]{ \|  #1 \|}
\def\expect{\mathbb E}
\def\bull{$\bullet$\ }
\def\det{\operatorname{det}}
\def\Det{\operatorname{Det}}
\def\multiR{\mathbf R}
\def\bestA{\mathbf A}
\def\Apq{\mathbf A_{p,q}}
\def\Apqr{\mathbf A_{p,q,r}}
\def\rank{\mathbf r}
\def\diameter{\operatorname{diameter}}
\def\bp{\mathbf p}
\def\bff{\mathbf f}
\def\bg{\mathbf g}
\def\essd{\operatorname{essential\ diameter}}

\def\mab{M}
\def\t2{\tfrac12}

\newcommand{\abr}[1]{ \langle  #1 \rangle}

\def\doublesymm{\natural}

\newcommand{\Norm}[1]{ \Big\|  #1 \Big\| }
\newcommand{\set}[1]{ \left\{ #1 \right\} }
\def\one{{\mathbf 1}}
\newcommand{\modulo}[2]{[#1]_{#2}}

\def\barrier{\medskip\hrule\hrule\medskip}

\def\scriptf{{\mathcal F}}
\def\scripts{{\mathcal S}}
\def\scriptq{{\mathcal Q}}
\def\scriptg{{\mathcal G}}
\def\scriptm{{\mathcal M}}
\def\scriptb{{\mathcal B}}
\def\scriptc{{\mathcal C}}
\def\scriptt{{\mathcal T}}
\def\scripti{{\mathcal I}}
\def\scripte{{\mathcal E}}
\def\scriptv{{\mathcal V}}
\def\scriptw{{\mathcal W}}
\def\scriptu{{\mathcal U}}
\def\scripta{{\mathcal A}}
\def\scriptr{{\mathcal R}}
\def\scripto{{\mathcal O}}
\def\scripth{{\mathcal H}}
\def\scriptd{{\mathcal D}}
\def\scriptl{{\mathcal L}}
\def\scriptn{{\mathcal N}}
\def\scriptp{{\mathcal P}}
\def\scriptk{{\mathcal K}}
\def\scriptP{{\mathcal P}}
\def\scriptj{{\mathcal J}}
\def\scriptz{{\mathcal Z}}
\def\frakv{{\mathfrak V}}
\def\frakG{{\mathfrak G}}
\def\frakA{{\mathfrak A}}
\def\frakB{{\mathfrak B}}
\def\frakC{{\mathfrak C}}

\author{Michael Christ}
\address{
        Michael Christ\\
        Department of Mathematics\\
        University of California \\
        Berkeley, CA 94720-3840, USA}
\email{mchrist@math.berkeley.edu}
\thanks{Research supported in part by NSF grant DMS-0901569.}


\date{
September 23, 2013.}

\title {Near equality in the Riesz-Sobolev inequality}

\begin{abstract}
The Riesz-Sobolev inequality  provides a sharp upper bound
for a trilinear expression involving convolution of indicator functions of sets.
Equality is known to hold only for indicator functions of appropriately
situated intervals. We characterize ordered triples of subsets of $\reals^1$ 
that nearly realize equality, with quantitative bounds of power law form with the optimal exponent.
\end{abstract}
\maketitle

\section{Introduction}

Denote by $|A|$ the Lebesgue measure of a set $A\subset\reals^d$, and by $\one_A$ the indicator function of $A$. 
By $\langle f,g\rangle$ we will mean $\int_{\reals^d} fg\,dm$ where $m$ is Lebesgue measure.
Let $f^\star$ denote the radially symmetric nonincreasing rearrangement of $f$,
whose definition is reviewed below.

Let $f,g,h: \reals^d\to[0,\infty)$ be nonnegative Lebesgue measurable functions 
such that $|\set{x: f(x)>t}|<\infty$ for all $t>0$, and likewise for $g,h$.
The inequality of Riesz and Sobolev \cite{riesz},\cite{sobolev} states that
\begin{equation}\label{eq:RSforfns} \langle f*g,h\rangle \le \langle f^\star*g^\star,h^\star\rangle.\end{equation}
Of particular importance is the special case of 
indicator functions of measurable sets with finite Lebesgue measures, for which the inequality becomes
\begin{equation}\label{eq:RS} \langle \one_A*\one_B,\one_C\rangle \le
\langle \one_{A^\star}*\one_{B^\star},\one_{C^\star}\rangle  \end{equation}
where $A^\star,B^\star,C^\star$ are balls centered at $0$ whose measures equal the measures of $A,B,C$ respectively.
The right-hand side is of course a function of $\set{|A|,|B|,|C|}$ alone.
This foundational inequality for indicator functions of sets directly implies the formally more general 
version \eqref{eq:RSforfns} for nonnegative functions.
The multiplicity function $(\one_A*\one_B)(x)$ is a continuum measurement of the number of ways in which $x$ is represented
in the form $a+b$ with $(a,b)\in A\times B$; $\langle \one_A*\one_B,\,\one_C\rangle$
represents the total number of ways in which elements of $C$ can be so represented. 

Inverse theorems that characterize cases of equality in inequalities \eqref{eq:RS} and/or \eqref{eq:RSforfns}
have been a useful tool in the analysis of extremizers of other
inequalities. For instance, one element of Lieb's characterization \cite{liebHLS} of extremizers of the
Hardy-Littlewood-Sobolev inequality was the fact that
if $h=h^\star$, and if $h^\star$ is positive and strictly decreasing along rays emanating from the origin,
then equality holds in \eqref{eq:RSforfns} only if $f=f^\star$ and $g=g^\star$ up to suitable translations.
See for instance Theorem~3.9 in \cite{liebloss}.
More recently, this author in \cite{christradon} used a sharper inverse theorem of Burchard \cite{burchard}
concerning \eqref{eq:RS} to determine all extremizers of an inequality for the Radon transform. 
The one-dimensional case of Burchard's theorem states that if $(A,B,C)$ is an ordered 
triple of subsets of $\reals^1$ whose measures satisfy a natural admissibility
condition introduced below, 
then equality holds in \eqref{eq:RS} only if
$A,B,C$ coincide with intervals, up to null sets. 
Equality also forces the centers $a,b,c$ of these intervals to satisfy $a+b=c$.
These conditions together are also sufficient for equality.

An associated question is what properties $(A,B,C)$ must have if the condition of exact equality
in \eqref{eq:RS} is relaxed to near equality. 
If $\langle \one_A*\one_B,\,\one_C\rangle$ is nearly maximal among such expressions for all ordered triples 
with $(|A|,|B|,|C|)$ specified, must $(A,B,C)$ be close to some extremizing triple? In what metric must
it be close? How close? One seeks a compactness principle, modulo the action of a noncompact symmetry group.
This paper is one of a series devoted to these stability questions, for 
functionals and inequalities that are governed by the Abelian group structure of Euclidean space and have the group of all
affine automorphisms as an underlying symmetry group. One of our principal tools, an additive combinatorial
inverse theorem, was originally developed
in the context of finite sets and discrete groups, but has proved effective in the continuum as well. 

A weak theorem describing triples realizing near equality in \eqref{eq:RS} was established in \cite{christRS1}.  
This seems to have been the first usage of the inverse theorem in the context of such analytic inequalities for Euclidean space.
It served as the central element of a characterization \cite{christyoungest} of those functions which 
nearly extremize Young's convolution inequality in $\reals^d$, for arbitrary $d$.
Notwithstanding its adequacy for that application, this inverse theorem suffers from severe limitations: 
It is only applicable if one is given two sets $C_1,C_2$ such that both triples $(A,B,C_i)$ achieve near equality; 
$|C_2|$ must be very nearly equal to $3|C_1|$; and its conclusion is of ``little $o$'' form.

In this paper we establish a more definitive result for $d=1$, with natural hypotheses and a quantitative conclusion
of power law type in which the principal exponent is the best possible.
However, the analysis developed here exploits structural aspects of the one-dimensional case which do not extend to higher dimensions.
We plan to address the dimensional limitation in subsequent work, by combining the one-dimensional result
with other arguments, rather than by extending the one-dimensional method of proof.
Some progress on related problems in higher dimensions was made by this author in 
\cite{christyoungest}, \cite{christbmtwo}, \cite{christbmhigh} 
and by Figalli and Jerison \cite{figallijerison}. The latter authors have obtained a power law type bound, in an
analogous result for the Brunn-Minkowski inequality.

\section{Main Theorem}

Let $(A_1,A_2,A_3)$ be an ordered triple of measurable subsets of $\reals^1$ with finite, positive Lebesgue measures.
No inverse theorem is possible without a natural hypothesis, called admissibility by Burchard \cite{burchard}. 
$(A_1,A_2,A_3)$ is said to be admissible if $|A_i|+|A_j|\ge |A_k|$ for every permutation $(i,j,k)$ of $(1,2,3)$,
and to be strictly admissible if $|A_i|+|A_j|> |A_k|$ for every permutation. These definitions are independent
of the ordering.
Formulation of our main theorem requires the following more quantitative version of strict admissibility. 

\begin{definition}
Let $\eta\in(0,1]$. Let $A_j$ be measurable subsets of $\reals^1$ satisfying $0<|A_j|<\infty$.
The ordered triple $(A_1,A_2,A_3)$ is $\eta$--strictly admissible if
\begin{equation} |A_i|+|A_j| \ge  |A_k|+ \eta\max(|A_1|,|A_2|,|A_3|) \end{equation}
for every permutation $(i,j,k)$ of $(1,2,3)$. 
\end{definition}
An immediate consequence of $\eta$--strict admissibility is mutual comparability of the measures of the sets in question: 
\begin{equation} \label{eq:comparability} \min_m|A_m| \ge \eta\max_n|A_n|.  \end{equation}
This will be proved below.

$S\bigtriangleup T$ will denote the symmetric difference between sets $S,T$,
and $-S=\set{-s: s\in S}$.
Our main result is as follows.
\begin{theorem} \label{mainthm}
There exists an absolute constant $K<\infty$ for which the following holds.
Let $\eta\in(0,1]$.
Let $(A,B,C)$ be an $\eta$--strictly admissible ordered triple of 
measurable subsets of $\reals^1$ with finite, positive Lebesgue measures.
If
\begin{equation} 
\label{nearlysharp}
\langle \one_A*\one_B,\,\one_{C}\rangle \ge 
\langle \one_{A^\star}*\one_{B^\star},\,\one_{C^\star}\rangle -\eps \max(|A|,|B|,|C|)^2
\end{equation}
and if 
$\eps \le K^{-1} \eta^{4}$ 
then there exist intervals $I,J,L\subset\reals$ such that
\begin{equation} \label{eq:conclusion1} |A\bigtriangleup I| \le K\eta^{-1}\eps^{1/2}\max(|A|,|B|,|C|) \end{equation}
and the corresponding upper bounds hold for $|J\bigtriangleup B|$ and $|L\bigtriangleup C|$.
The centers $a,b,c$ of $I,J,L$ satisfy
\begin{equation}\label{eq:conclusion2} |a+b-c|\le K\eta^{-2}\eps^{1/4} \max(|A|,|B|,|C|).  \end{equation}
\end{theorem}

The exponent $\tfrac12$ in \eqref{eq:conclusion1} is optimal.
Indeed, if $A,B$ are intervals centered at $0$, if $C = [-\gamma,\gamma]\cup[\gamma+\delta,\gamma+2\delta]$,
and if $(A,B,[-\gamma,\gamma])$ is strictly admissible then for sufficiently small $\delta$,
\eqref{nearlysharp} holds with $\eps\asymp \delta^2$.

The Riesz-Sobolev inequality can be viewed as a statement about additive combinatorics.
The quantity $\langle \one_A*\one_B,\one_C\rangle$ is interpreted as the number of ordered
pairs $(a,b)\in A\times B$ for which the sum $a+b$ lies in $C$. 

Theorem \ref{mainthm} can be interpreted as a sharpening of the Riesz-Sobolev inequality, in the following way.
The infimum in the following inequality is taken over all bounded intervals $I\subset\reals$.
\begin{theorem}
There exists a constant $c_0>0$ such that for any
$\eta\in(0,1]$  and any sets $A,B,C$ satisfying the hypotheses of Theorem~\ref{mainthm},
\begin{equation}\langle \one_A*\one_B,\one_C\rangle 
\le 
\langle \one_{A^\star}*\one_{B^\star},\one_{C^\star}\rangle  
- c_0 \eta^2 \inf_I  |A\bigtriangleup I|^2.
\end{equation}
\end{theorem}
The hypotheses are symmetric in $(A,B,C)$ in a natural way, so in the second term on the
left-hand side of the inequality, $A$ can equally be replaced by $B$ or by $C$.

The Riesz-Sobolev inequality is very closely related to another one, the KPRGT inequality.  In Theorem~\ref{thm:KT} we formulate
an analogous inverse result for the KPRGT inequality, and deduce it as a corollary of Theorem~\ref{mainthm}.

The author thanks Marcos Charalambides and Ed Scerbo for proofreading
and for valuable suggestions which have improved the exposition, and Terence Tao for 
calling his attention to the KPRGT inequality.

\section{Outline and notations}

The leading idea in the proof, as in \cite{christRS1}, is to relate near equality in the Riesz-Sobolev inequality
to near equality in the Brunn-Minkowski inequality, for which a characterization is
already available. The essential difference between the two situations is that for Brunn-Minkowski, 
one is given that $a+b\in C$ for every ordered pair $(a,b)\in A\times B$, whereas for Riesz-Sobolev
it is given that $a+b\in C$ for a subset of $A\times B$ whose complement
has measure comparable to that of $A\times B$, even in cases of exact equality.

The superlevel sets 
\begin{equation} S_{A,B}(t)=\set{x\in\reals^1: (\one_A*\one_B)(x)>t} \end{equation}
play a central role in the anaysis.  There are multiple steps, organized as follows although not in this order.
\begin{enumerate}
\item
If an ordered triple $(A,B,C)$ nearly attains equality in the Riesz-Sobolev inequality,
then $C$ nearly coincides with the superlevel set $S_{A,B}(\alpha)$ for a certain parameter $\alpha$
which depends only on $(|A|,|B|,|C|)$.
Moreover, the ordered triple $(A,B,S_{A,B}(\alpha))$ also nearly attains equality,
so that $(A,B,C)$ can be replaced by $(A,B,S_{A,B}(\alpha))$.
\item
Superlevel sets associated to convolutions of indicator functions of arbitrary sets
satisfy an additive inclusion relation: the difference set 
$S_{A,B}(\alpha)-S_{A,B}(\beta)$ is contained in  $S_{A, -A}(\alpha+\beta-|B|)$.
\item
An inverse theorem associated to the one-dimensional Brunn-Minkowski inequality
asserts that if 
$|\scripta+\scriptb|$ is nearly equal to $|\scripta|+|\scriptb|$, then $\scripta, \scriptb$ nearly coincide with intervals.
Thus in order to show that $S_{A,B}(\alpha)$ and hence $C$ are nearly equal to intervals, it suffices to show that 
the measure of the difference set $S_{A,B}(\alpha)-S_{A,B}(\alpha)$
is only slightly greater than twice the measure of $S_{A,B}(\alpha)$.
By the inclusion relation, this in turn would follow from the same upper bound for $|S_{A,-A}(2\alpha-|B|)|$. 
\item
The Riesz-Sobolev inequality is equivalent to another inequality, which we call the (sharpened) KPRGT inequality. 
Whereas the Riesz-Sobolev upper bound is expressed in terms of $|A|,|B|,|S_{A,B}(\tau)|$,
the KPRGT bound is expressed in terms of $|A|,|B|,\tau$.
Therefore it is potentially possible to study whether such a triple of sets nearly extremizes the
KPRGT inequality, without knowing $|S_{A,B}(\tau)|$.

\item
If $(A,B,S_{A,B}(\alpha))$ nearly realizes equality in the Riesz-Sobolev inequality,
then the ordered triple $(A,-A, S_{A,-A}{(2\alpha-|B|)})$ nearly achieves equality in the KPRGT inequality 
--- but our argument for this implication applies
only under the excruciatingly restrictive extra hypothesis that $|A|=|B|$.
This step uses the inclusion relation involving differences of superlevel sets, and the Brunn-Minkowski inequality
to obtain {\em lower} bounds for measures of these differences. 
\item
Whenever $(A,B,S_{A,B}(\tau))$ nearly extremizes the KPRGT inequality, 
a nearly tight bound must hold for $|S_{A,B}(\tau)|$.
In the present context, this is the desired upper bound for $|S_{A,-A}(2\alpha-|B|)|$.
\item 
By the inverse theorem, $S_{A,B}(\alpha)$ and hence $C$ nearly coincide with an interval, 
concluding the proof (for $C$) when $|A|=|B|$.
\item
An alteration procedure makes it possible to replace sets with certain subsets, 
without sacrificing the hypothesis of near equality in the Riesz-Sobolev inequality.
This is used to replace $A,B$ by subsets with equal measures, making the special case treated above applicable.
\item The alteration procedure is sufficiently flexible to give rise to a rich family of subsets of $A$.
Near coincidence of all subsets in such a family with intervals is shown to imply
near coincidence of $A$ itself with a larger interval.
\end{enumerate}


Our exposition includes largely self-contained proofs of the Riesz-Sobolev, KPRGT, and sharpened KPRGT inequalities,
of the additive combinatorial inverse theorem (in the relevant continuum version) of Fre{\u\i}man
that is a keystone of the analysis, and of the equivalence of two formulations of the Riesz-Sobolev inequality.

\medskip
Symmetric nonincreasing rearrangements are defined as follows.
Let $|S|$ denote the Lebesgue measure of $S\subset\reals^d$.
If $S$ is Lebesgue measurable and $0<|S|<\infty$, then
$S^\star$ denotes the open ball $B$ centered at $0\in\reals^d$ which satisfies $|B|=|S|$.
If $f:\reals^d\to[0,\infty)$ is a Lebesgue measurable function for which
$|\set{x: f(x)>t}|$ is finite for any $t>0$,
then $f^\star$ is defined to be the unique radially symmetric function
such that $r\mapsto f^\star(rx)$ is a nonincreasing function of $r>0$ for each $0\ne x\in\reals^d$,
$|\set{x: f^\star(x)>t}|= |\set{x: f(x)>t}|$ for all $t>0$,
and $r\mapsto f^\star(rx)$ is right continuous for each $0\ne x$.


\medskip To prove the claim \eqref{eq:comparability} made above, 
let $(A,B,C)$ be an $\eta$--strictly admissible ordered triple and assume without loss of generality that $|A|\ge |B|\ge |C|$.
The inequality to be proved is then that $|C|\ge \eta |A|$. It is given that
$|B|+|C|\ge |A|+\eta |A|$, so $|C|\ge \eta |A|+(|A|-|B|)\ge \eta |A|$. \qed

\section{The Riesz-Sobolev inequality recast}

To any Lebesgue measurable sets $A,B\subset\reals^1$ with finite Lebesgue measures
are associated the superlevel sets 
\begin{equation} S_{A,B}(t)=\set{x: (\one_A*\one_B)(x)>t}.\end{equation}
These sets are open since $\one_A*\one_B$ is a continuous function. 

For any two bounded intervals $I,J$, centered at $0$  
\begin{equation}
(\one_I*\one_{J})(x) = 
\begin{cases} 
\min(|I|,|J|) & \text{ if } 2|x|\le 
\big|\,|I|-|J|\,\big| 
\\ \tfrac12 |I|+ \tfrac12 |J|-|x| & \text{ if } \big|\,|I|-|J|\,\big| \le 2|x|\le |I|+|J|
\\ 0 & \text{ if } 2|x|\ge |I|+|J|.
\end{cases}
\end{equation}
Thus
\begin{equation}
|S_{I,J}(t)| = 
\begin{cases} 
0 & \text{ if } t\ge\min(|I|,|J|) 
\\ |I|+|J|-2t & \text{ if } 0\le t<\min(|I|,|J|).
\end{cases}
\end{equation}

The Riesz-Sobolev inequality for $\reals^1$ states that
\begin{equation}
\int_E (\one_A*\one_B) \le \int_{E^\star} (\one_{A^\star}*\one_{B^\star})
\end{equation}
for any sets $A,B,E\subset\reals^1$ with finite Lebesgue measures.
A proof is sketched in \S\ref{section:truncation}.
The right-hand side can be expressed in terms of $|A|,|B|,|E|$ using the formulas above.

The following notation will be used throughout the discussion.
\begin{definition}
For any Lebesgue measurable sets $A,B,C\subset\reals^d$
with finite Lebesgue measures,
\begin{equation}
\scriptd(A,B,C) = 
\langle \one_{A^\star}*\one_{B^\star},\,\one_{C^\star}\rangle
-\langle \one_A*\one_B,\,\one_C\rangle
\end{equation}
\end{definition}
The Riesz-Sobolev inequality states that $\scriptd(A,B,C)\ge 0$ for all $(A,B,C)$.

If \[ \big|\,|A|-|B|\,\big| \le|E|\le|A|+|B| \]
and if $\sigma\in[0,\min(|A|,|B|)]$ is defined by
\begin{equation}|E|=|A|+|B|-2\sigma\end{equation}
then
\begin{equation} \langle \one_{A^\star}*\one_{B^\star},\,\one_{E^\star}\rangle
= |A|\cdot|B|-\sigma^2 \end{equation}
and the Riesz-Sobolev inequality becomes
\begin{equation} \label{anotherRS} \langle \one_A*\one_B,\,\one_C\rangle \le |A|\cdot|B|-\sigma^2
 = |A|\cdot|B|-\tfrac14(|A|+|B|-|E|)^2.  \end{equation}
For $|E|>|A|+|B|$ the Riesz-Sobolev inequality for $\reals^1$
states the trivial upper bound $\int_E (\one_A*\one_B)\le |A|\cdot|B|$.
For $|E|< \big|\,|A|-|B|\,\big| $, it gives the 
also trivial upper bound $|E|\min(|A|,|B|)$.

The identity
\begin{equation} \label{slsetintegral}
 \int_{S_{A,B}(\tau)} (\one_A*\one_B)
= \tau|S_{A,B}(\tau)| + \int_{\tau}^{\infty} |S_{A,B}(t)|\,dt
\end{equation}
will be useful throughout our analysis.
It allows one to express the $\reals^1$ Riesz-Sobolev inequality, for the special case when $C$ is a superlevel
set of $\one_A*\one_B$, in the form

\begin{lemma}
Let $A,B\subset\reals^1$ be Lebesgue measurable sets with finite, positive measures.
Let $\tau\in[0,\min(|A|,|B|)]$. 
Define $\sigma$ by \begin{equation} |S_{A,B}(\tau)| = |A|+|B|-2\sigma.\end{equation}
If 
\begin{equation} \label{eq:naturalrange}  \big|\,|A|-|B|\,\big| \le |S_{A,B}(\tau)|\le |A|+|B| \end{equation}
then
\begin{equation} \label{RSrestated} 
\tau|S_{A,B}(\tau)| + \int_{\tau}^{\infty} |S_{A,B}(t)|\,dt 
\le |A|\cdot|B|-\sigma^2
\end{equation}
\end{lemma}
The assumption \eqref{eq:naturalrange} is equivalent to $\sigma\in[0,\min(|A|,|B|)]$.
We will show in \S\ref{section:equivalence}
how the Riesz-Sobolev inequality for general sets $C$ can in turn be deduced from this lemma.

\section{Approximation by superlevel sets}
If $A,B$ are given then in order to maximize $\langle \one_A*\one_B,\,\one_C\rangle$
over all sets $C$ of specified measure, $C$ should be chosen to be a superlevel set of that measure,
provided such a superlevel set exists. The purpose of this section is to show that
if $(A,B,C)$ is a nearly extremizing ordered triple, 
then $C$ must nearly coincide with some superlevel set $S_{A,B}(t)$,
and moreover $|S_{A,B}(t)|$ must be nearly equal to $|A|+|B|-2t$.
This was shown in \cite{christRS1}, but we give more precise bounds here.
\begin{lemma} \label{lemma:Staulowerbound}
Let $A,B,E\subset\reals^1$ be Lebesgue measurable sets of finite, positive measures. Suppose that
\begin{equation} 
 \big|\,|A|-|B|\,\big|  + 2\scriptd(A,B,E)^{1/2}< |E| < |A|+|B|-2\scriptd(A,B,E)^{1/2}.
\end{equation}
Define $\tau$ by $|E|=|A|+|B|-2\tau$.  Then 
\begin{gather} \label{ineq:symmetricdifferencebound}
|S_{A,B}(\tau)\bigtriangleup E| \le 4\scriptd(A,B,E)^{1/2}
\\ \big|\, |S_{A,B}(\tau)|-(|A|+|B|-2\tau) \,\big| \le 4\scriptd(A,B,E)^{1/2}.
\label{ineq:sigmanearlytau}
\end{gather}
\end{lemma}

\begin{proof}[Proof of Lemma~\ref{lemma:Staulowerbound}]
To simplify notation write $S = S_{A,B}(\tau)$,
$\lambda =  \big|\,|A|-|B|\,\big| $, and $\scriptd=\scriptd(A,B,E)$.
By definition of $\scriptd(A,B,E)$,
\[ \langle \one_A*\one_B,\one_E\rangle = |A|\cdot|B|-\tfrac14(|A|+|B|-|E|)^2-\scriptd.  \]
Since $|E|>\lambda +2\scriptd^{1/2}$, if $|E\setminus S|$ were strictly greater
than $2\scriptd^{1/2}$ then there would exist a measurable set $T$ satisfying 
$E\cap S\subset T\subset E$ with $|T|\ge\lambda$ and $|E\setminus T|>2\scriptd^{1/2}$.  
Indeed, if $|E\cap S|\ge\lambda$ choose $T=E\cap S$. Otherwise choose any measurable set $T$
satisfying $|T|=\lambda$ with $E\cap S\subset T\subset E$.

Because $|T|\ge\lambda = \big|\,|A|-|B|\,\big|$,
\begin{align*}
\langle \one_A*\one_B,\one_E\rangle
&= \langle \one_A*\one_B,\one_{T}\rangle
+ \langle \one_A*\one_B,\one_{E\setminus T}\rangle
\\ &\le |A|\cdot|B|
-\tfrac14(|A|+|B|-|T|)^2 + \tau|E\setminus T|
\\ & =  |A|\cdot|B|
-\tfrac14 (2\tau+|E\setminus T|)^2 +\tau|E\setminus T|
\\ & =  |A|\cdot|B|
-\tau^2 -\tfrac14 |E\setminus T|^2
\\ &= 
|A|\cdot|B|-\tfrac14(|A|+|B|-|E|)^2 
- \tfrac14 |E\setminus T|^2.
\end{align*}
Thus $|E\setminus T| \le 2\scriptd^{1/2}$, which is a contradiction.

To establish an upper bound for $|S\setminus E|$, consider 
any set $T$ satisfying $E\subset T \subset E\cup S$ with $|T|\le |A|+|B|$.
Since $\one_A*\one_B > \tau$ at each point of $T\setminus E$, 
\begin{align*}
\langle \one_A*\one_B,\one_{T}\rangle
&\ge \langle \one_A*\one_B,\one_{E}\rangle +\tau|T\setminus E|
\\ &=  |A|\cdot|B|-\tfrac14(|A|+|B|-|E|)^2 -\scriptd
+\tau|T\setminus E|.
\end{align*}
Bearing in mind that $|T\setminus E|=|T|-|E|$,
\begin{align*}
|A|\cdot|B| &-\tfrac14(|A|+|B|-|E|)^2 +\tau|T\setminus E|
\\&=  |A|\cdot|B|-\tfrac14(|A|+|B|-|E|)^2 
+\tfrac	12(|A|+|B|-|E|)(|T|-|E|) 
\\&=  |A|\cdot|B|-\tfrac14(|A|+|B|-|T|)^2 
-\tfrac12(|A|+|B|-|T|)(|T|-|E|) -\tfrac14(|T|-|E|)^2
\\& \qquad +\tfrac12(|A|+|B|-|E|)(|T|-|E|) 
\\&=  |A|\cdot|B|-\tfrac14(|A|+|B|-|T|)^2 
+ \tfrac12(|T|-|E|)^2 -\tfrac14(|T|-|E|)^2 
\\ &=  |A|\cdot|B|-\tfrac14(|A|+|B|-|T|)^2 + \tfrac14|T\setminus E|^2.
\end{align*}
Thus
\begin{equation*}
\langle \one_A*\one_B,\one_{T}\rangle
\ge |A|\cdot|B|-\tfrac14(|A|+|B|-|T|)^2 + \tfrac14|T\setminus E|^2 -\scriptd.
\end{equation*}
Since the Riesz-Sobolev inequality guarantees that
$\langle \one_A*\one_B,\one_{T}\rangle$ cannot exceed the quantity $|A|\cdot|B|-\tfrac14(|A|+|B|-|T|)^2$,
it follows that
\[ \tfrac14 |T\setminus E|^2\le \scriptd.  \]

If $x\le y$ are numbers, $\scripta$ is a measurable set, and $|\scriptb|\le x$ for every
measurable subset $\scriptb\subset\scripta$ satisfying $|\scriptb|\le y$, then $|\scripta|\le x$.
From this principle and the preceding inequality, 
since $|E|\le |A|+|B|-2\scriptd^{1/2}$ it now follows that $|S\setminus E|\le 2\scriptd^{1/2}$. 

Summing the bounds for $|E\setminus S|$ and $|S\setminus E|$ demonstrates that
$S=S_{A,B}(\tau)$ satisfies $|S\bigtriangleup E|\le 4\scriptd^{1/2}$.
Since $\tau$ is defined so that $|E|=|A|+|B|-2\tau$,
\begin{equation*} \big|\,|S_{A,B}(\tau)|-(|A|+|B|-2\tau)|\,\big|
 = \big|\,|S_{A,B}(\tau)|-|E|\,\big|\le |S_{A,B}(\tau)\bigtriangleup E| \le 4\scriptd^{1/2}, \end{equation*}
which is the final conclusion \eqref{ineq:sigmanearlytau} of the lemma.
\end{proof}

Lemma~\ref{lemma:Staulowerbound} can be reinterpreted as a refinement of the $\reals^1$ Riesz-Sobolev inequality.
\begin{lemma}
Under the hypotheses of Lemma~\ref{lemma:Staulowerbound},
defining $\tau_E$ by the relation \begin{equation} |E|=|A|+|B|-2\tau_E,\end{equation}
one has
\begin{multline} 
\langle \one_A*\one_B,\one_E\rangle
+ \tfrac1{16}\Big(|S_{A,B}(\tau_E)|-(|A|+|B|-2\tau_E)\Big)^2 
\le \langle \one_{A^\star}*\one_{B^\star},\one_{E^\star}\rangle.  \end{multline}
\end{lemma}

It follows directly from the upper bound for $|S_{A,B}(\tau)\bigtriangleup E|$
that the triple $(A,B,S_{A,B}(\tau))$ nearly attains
equality in the Riesz-Sobolev inequality, with a discrepancy majorized by a constant multiple of
$\scriptd(A,B,E)^{1/2} \max(|A|,|B|,|E|)$. The next lemma gives a better bound, which will be essential 
in the attainment of the optimal exponent in Theorem~\ref{mainthm}.

\begin{lemma} \label{lemma:november}
Let $(A,B,E)$ be an ordered triple of sets satisfying
the hypotheses of Lemma~\ref{lemma:Staulowerbound}. Define $\tau$ by $|E|=|A|+|B|-2\tau$.
Then the superlevel set $S_{A,B}(\tau)$ satisfies
\begin{gather}  
\scriptd(A,B,S_{A,B}(\tau)) \le \scriptd(A,B,E).
\end{gather}
\end{lemma}

\begin{proof}
Write $\scriptd=\scriptd(A,B,E)$.
One has 
\[  \big|\,|A|-|B|\,\big|  < |S_{A,B}(\tau)| < |A|+|B| \]
since it was shown above that $|E\setminus S_{A,B}(\tau)|\le 2\scriptd^{1/2}$, 
and likewise for $|S_{A,B}(\tau)\setminus E|$.

The convolution $\one_A*\one_B$ satisfies
$(\one_A*\one_B)(x)\ge \tau$ for $x\in S_{A,B}(\tau)\setminus E$,
and $\le\tau$ for $x\in E\setminus S_{A,B}(\tau)$.  Thus 
\begin{align*}  \langle \one_A*\one_B,\one_{S_{A,B}(\tau)}\rangle 
&\ge \langle \one_A*\one_B,\one_E \rangle 
+ \tau|S_{A,B}(\tau)\setminus E| - \tau|E\setminus S_{A,B}(\tau)| 
\\ & = \langle \one_A*\one_B,\one_E \rangle + \tau\big( |S_{A,B}(\tau)|-|E|\big)
\\ & = |A|\cdot|B|-\tau^2-\scriptd + \tau\big( |S_{A,B}(\tau)|-|E|\big).
\end{align*}
Defining $\sigma$ by $|S_{A,B}(\tau)| = |A|+|B|-2\sigma$, this can be rewritten  
\begin{align*}  \langle \one_A*\one_B,\one_{S_{A,B}(\tau)}\rangle 
& \ge |A|\cdot|B|-\sigma^2 + \sigma^2 -\tau^2 -\scriptd
 + \tau\big( |S_{A,B}(\tau)|-|E|\big)
\\ &  =  |A|\cdot|B|-\sigma^2 + \sigma^2 -\tau^2 -\scriptd + \tau\big( -2\sigma+2\tau\big)
\\ & = |A|\cdot|B|-\sigma^2 -\scriptd + (\sigma-\tau)^2
\\ & \ge  |A|\cdot|B|-\sigma^2 -\scriptd
\\ & =  \langle \one_{A^*}*\one_{B^*},\,\one_{S_{A,B}(\tau)^*}\rangle -\scriptd.
\end{align*}
\end{proof}

\section{The KPRGT inequality}

The Riesz-Sobolev inequality has a close relative, which we will call the KPRGT inequality. 
Important contributions to its theory were made by
Kemperman \cite{kemperman1964},
Pollard \cite{pollard},
Ruzsa \cite{ruzsa},
and Green and Ruzsa \cite{greenruzsa};
the form most directly relevant to our discussion was established by Tao \cite{taospending}.
The KPRGT inequality, in the version of \cite{taospending}, states that for any compact connected Abelian group $G$
equipped with a translation-invariant Borel probability measure $\mu$, 
for any Borel sets $A,B\subset G$ and any $\tau\in[0,\min(\mu(A),\mu(B))]$,
\begin{equation} \label{eq:kprgt} \int_G \min\big((\one_A*\one_B)(x),\,\tau\big)\,d\mu(x) 
\ge \tau\min\big(\mu(A)+\mu(B)-\tau,\,1\big).  \end{equation}

This has as a corollary the KPRGT inequality for $\reals^d$:
\begin{equation} \label{eq:kprgtRd} 
\int_{\reals^d} \min\big((\one_A*\one_B)(x),\,\tau\big)\,dx \ge \tau(|A|+|B|-\tau)  \end{equation}
provided $0\le\tau\le\min(|A|,|B|)$.
To deduce \eqref{eq:kprgtRd} from \eqref{eq:kprgt}, first consider bounded sets
apply \eqref{eq:kprgt} to their images under the quotient map $\reals^d\to\reals^d/\integers^d$ 
defined by $x\mapsto \eps x$ modulo $\integers^d$, for sufficiently small $\eps>0$. 
The case of unbounded sets with finite Lebesgue measures follows by a limiting argument.
The details are omitted, since an alternative proof will be provided below. 

Since $\int_{\reals^d}(\one_A*\one_B)\,dx = |A|\cdot|B|$, and
\[ \int_{\reals^d} f = \tau|\set{f>\tau}| + \int_{f\le \tau} f + \int_{f>\tau} (f-\tau)  
= \int_{\reals^d} \min(f,\tau) +\int_{f>\tau} (f-\tau) \]
for any nonnegative measurable function $f$ and any $\tau\ge 0$, \eqref{eq:kprgtRd} is equivalent 
for $0\le\tau\le\min(|A|,|B|)$ to
\begin{equation} 
\int_{S_{A,B}(\tau)} \big((\one_A*\one_B)(x)-\tau\big)\,dx \le |A|\cdot|B|-\tau(|A|+|B|-\tau) = (|A|-\tau)(|B|-\tau),
\end{equation}
which can also be equivalently written as
\begin{equation} \label{eq:altdistfinal}
\int_\tau^{\infty} |S_{A,B}(t)|\,dt \le (|A|-\tau)(|B|-\tau)
\ \ \text{for $0\le\tau\le\min(|A|,|B|)$}
\end{equation}
for $A,B\subset\reals^d$. 

This is not a sharp inequality for $d>1$, and we restrict the discussion henceforth to $d=1$. 
We will refer to \eqref{eq:altdistfinal} as the KPRGT inequality. 
It should be compared with the Riesz-Sobolev inequality \eqref{RSrestated}, which provides an upper bound for the sum of
the left-hand side of \eqref{eq:altdistfinal} plus $\tau|S_{A,B}(\tau)|$. 

Both the KPRGT \eqref{eq:altdistfinal} and
Riesz-Sobolev \eqref{RSrestated} inequalities give integrated upper bounds for the measures of superlevel sets,
rather than any bound for any individual superlevel set.
Yet our present goal is a tight bound for $|S_{A,-A}(2\alpha-B|)|$ when $S_{A,B}(\alpha)$ is as in Lemma~\ref{lemma:Staulowerbound}.

\section{Connection between the Riesz-Sobolev and KPRGT inequalities}

The connection between these two inequalities can be expressed succinctly using the following variant
of $\scriptd(A,B,S_{A,B}(\tau))$.
\begin{definition}
For any sets $A,B$ and any real number $\tau\in[0,\min(|A|,|B|)]$, the deficit $\scriptd'(A,B,\tau)$ is
\begin{equation}
\scriptd'(A,B,\tau) 
= (|A|-\tau)(|B|-\tau) - \int_{\tau}^{\infty} |S_{A,B}(t)|\,dt.
\end{equation}
\end{definition}
The KPRGT inequality asserts simply that $\scriptd'(A,B,\tau)\ge 0$  for $0\le\tau\le\min(|A|,|B|)$.
The two quantities $\scriptd'(A,B,\tau)$ and $\scriptd(A,B,S_{A,B}(\tau))$ are related by the following identities.

\begin{lemma} \label{lemma:RSKidentity}
Let $A,B\subset\reals^1$ be measurable sets with finite, positive Lebesgue measures. 
Let $\tau\in [0,\min(|A|,|B|)]$, and suppose that
\begin{equation}\label{eq:RSKhypothesis} 
 \big|\,|A|-|B|\,\big| \le|S_{A,B}(\tau)| \le |A|+|B|.\end{equation}
Then
\begin{equation} \label{twoinequalityidentity} \scriptd'(A,B,\tau) = \scriptd(A,B,S_{A,B}(\tau)) + (\sigma-\tau)^2 \end{equation}
where $|S_{A,B}(\tau)|=|A|+|B|-2\sigma$.  

If $|S_{A,B}(\tau)|\ge |A|+|B|$ then
\begin{equation} \label{eq:RSKalt1}
\scriptd'(A,B,\tau) = \scriptd(A,B,S_{A,B}(\tau)) + \tau(|S_{A,B}(\tau)-|A|-|B|) + \tau^2.
\end{equation}

If $|S_{A,B}(\tau)|\le \min(|A|,|B|)$ then
\begin{equation} \label{eq:RSKalt2}
\scriptd'(A,B,\tau) = \scriptd(A,B,S_{A,B}(\tau)) + (|B|-\tau)(|A|-|S|-\tau).
\end{equation}
\end{lemma}

The expressions $\tau(|S_{A,B}(\tau)-|A|-|B|)$ in \eqref{eq:RSKalt1} and $(|B|-\tau)(|A|-|S|-\tau)$ in \eqref{eq:RSKalt2}
are nonnegative, under the indicated assumptions about $|S_{A,B}(\tau)|$.

\begin{proof}
Suppose that \eqref{eq:RSKhypothesis} holds.
Recall that under this assumption, $\scriptd=\scriptd(A,B,S_{A,B}(\tau))$ can be expressed as
\[ \scriptd = |A|\cdot|B|-\sigma^2 - \int_{S_{A,B}(\tau)}(\one_A*\one_B).  \]
By \eqref{slsetintegral},
\begin{align*}
\int_\tau^{\infty} |S_{A,B}(t)|\,dt 
&= \int_{S_{A,B}(\tau)}(\one_A*\one_B) - \tau|S_{A,B}(\tau)|
\\ &= |A|\cdot|B|-\sigma^2 -\scriptd -\tau (|A|+|B|-2\sigma) 
\\ &= (|A|-\tau)(|B|-\tau) -(\tau-\sigma)^2 -\scriptd.
\end{align*}

Next let $S = S_{A,B}(\tau)$ and suppose that $|S_{A,B}(\tau)|>|A|+|B|$. Then 
$\langle \one_{A^\star}*\one_{B^\star},\,\one_{S^\star}\rangle = |A|\cdot|B|$,
so $\scriptd(A,B,S_{A,B}(\tau)) = |A|\cdot|B| - \int_{S} \one_A*\one_B$.
Therefore
\begin{align*}
\scriptd'(A,B,\tau) 
& = (|A|-\tau)(|B|-\tau) - \int_\tau^{\infty} |S_{A,B}(t)|\,dt 
\\&= (|A|-\tau)(|B|-\tau) -\int_S (\one_A*\one_B) + \tau|S|
\\&= (|A|-\tau)(|B|-\tau) -|A|\cdot|B|+ \scriptd(A,B,S)  + \tau|S|
\\&= \tau(\tau-|A|-|B|) + \tau|S| +  \scriptd(A,B,S)  
\\&= \tau(|S|-|A|-|B|) + \tau^2 +  \scriptd(A,B,S).  
\end{align*}

If $|S|<\min(|A|,|B|)$ suppose without loss of generality that $|A|\ge |B|$. Then
$\langle \one_{A^\star}*\one_{B^\star},\,\one_{S^\star}\rangle = |B|\cdot|S|$,
so $\scriptd(A,B,S_{A,B}(\tau)) = |B|\cdot|S| - \int_{S} \one_A*\one_B$. Therefore as above,
\begin{align*}
\scriptd'(A,B,\tau) & = (|A|-\tau)(|B|-\tau) - \int_\tau^{\infty} |S_{A,B}(t)|\,dt 
\\&= (|A|-\tau)(|B|-\tau) -|S|\cdot|B|+ \scriptd(A,B,S)  + \tau|S|
\\&= \scriptd(A,B,S) + (|B|-\tau)(|A|-|S|-\tau).
\end{align*}
\end{proof}

In particular, $0\le \scriptd(A,B,S_{A,B}(\tau))\le\scriptd'(A,B,\tau)$ in all cases. 
Conversely, in the main case \eqref{eq:RSKhypothesis}, 
an inequality in the reverse direction holds provided that $|\sigma-\tau|$ can be suitably controlled.
Thus for superlevel sets, near equality in the KPRGT inequality implies near equality in the Riesz-Sobolev
inequality, while the reverse holds if a suitable upper bound is valid for $|\sigma-\tau|$.


The KPRGT inequality is sharp, in the sense that equality holds whenever $A,B$ are intervals and $0\le\tau\le \min(|A|,|B|)$.
Nonetheless, a yet sharper inequality is implicit in the identities of Lemma~\ref{lemma:RSKidentity}.

\begin{corollary} \label{cor:KTsharpened} \label{cor:twoway}
Let $A,B\subset\reals^1$ be measurable sets with finite, positive Lebesgue measures. 
Let $\tau\in [0,\min(|A|,|B|)]$, and suppose that
$ \big|\,|A|-|B|\,\big| \le|S_{A,B}(\tau)| \le |A|+|B|$.
Define $\sigma$ by $|S_{A,B}(\tau)|=|A|+|B|-2\sigma$.  Then
\begin{equation} \scriptd'(A,B,\tau) \ge 
(\sigma-\tau)^2 
= \tfrac14(|A|+|B|-|S_{A,B}(\tau)|)^2 
\end{equation}
if $ \big|\,|A|-|B|\,\big| \le |S_{A,B}(\tau)|\le |A|+|B|$ 
while
\begin{equation} \scriptd'(A,B,\tau) \ge 
\begin{cases}
\tau(|S|-|A|-|B|) + \tau^2 
&\text{if $|S_{A,B}(\tau)|\ge |A|+|B|$}
\\ (|B|-\tau)(|A|-|S|-\tau)
&\text{if $|S_{A,B}(\tau)|\le  \big|\,|A|-|B|\,\big| $.}
\end{cases}
\end{equation}
\end{corollary}
The first conclusion can be equivalently restated as
\begin{equation} \label{eq:altdistsharper} \int_\tau^\infty |S_{A,B}(t)|\,dt + (\sigma-\tau)^2 \le (|A|-\tau)(|B|-\tau)  \end{equation}
provided that $ \big|\,|A|-|B|\,\big| \le|S_{A,B}(\tau)| \le |A|+|B|$,
with corresponding restatements of the second and third conclusions.

\begin{proof}
According to the Riesz-Sobolev inequality, $\scriptd(A,B,S_{A,B}(\tau))\ge 0$.
Lemma~\ref{lemma:RSKidentity} thus gives all three conclusions.  \end{proof}
Alternatively, \eqref{eq:altdistsharper} can be deduced from the KPRGT inequality \eqref{eq:altdistfinal} itself
by application of \eqref{eq:altdistfinal} to $\int_\sigma^\infty |S_{A,B}(t)|\,dt$,
and comparison of this integral with $\int_\tau^\infty |S_{A,B}(t)|\,dt$, in analogy with the proof
of Lemma~\ref{lemma:Staulowerbound}.

We have shown via the  identity \eqref{twoinequalityidentity} 
that the sharpened KPRGT inequality \eqref{eq:altdistsharper} for $\reals^1$
is equivalent to the Riesz-Sobolev inequality for $\reals^1$, specialized to superlevel sets;
in particular, this provides an independent proof of the KPRGT inequality for $\reals^1$.
Inequality \eqref{eq:altdistsharper} is strictly sharper than \eqref{eq:altdistfinal}
unless $\sigma=\tau$, that is, unless $|S_{A,B}(\tau)| = |A|+|B| -2\tau$. 
Inequality \eqref{eq:altdistsharper} asserts in particular
that near inequality in the (unsharpened) KPRGT inequality in the form \eqref{eq:altdistfinal},
can only hold if $|S_{A,B}(\tau)|$ is nearly equal to $|A|+|B|-2\tau$.

A situation will arise below in which it will not be known that $|S_{A,B}(\tau)|\le |A|+|B|$. 
The following version of the KPRGT inequality (implicit in \eqref{eq:RSKalt1}) will be useful in that situation.
\begin{lemma} \label{lemma:kprtgoutsiderange} 
Suppose that $A,B\subset\reals^1$ are measurable sets with positive, finite Lebesgue measures.
Suppose that $0\le \tau < \max(|A|,|B|)$. If $|S_{A,B}(\tau)| \ge |A|+|B|$ then
\begin{equation} \int_\tau^{\infty} |S_{A,B}(t)|\,dt \le |A|\cdot|B|-\tau|S_{A,B}(\tau)|.  \end{equation}
\end{lemma}

\begin{proof}
Since
\[\tau|S_{A,B}(\tau)| + \int_\tau^{\infty} |S_{A,B}(t)|\,dt = \langle \one_A*\one_B,\one_{S_{A,B}(\tau)}\rangle \le |A|\cdot|B|,\]
the integral satisfies
\begin{align*}
\int_\tau^{\infty} |S_{A,B}(t)|\,dt 
&= \tau|S_{A,B}(\tau)| + \int_\tau^{\infty} |S_{A,B}(t)|\,dt -\tau|S_{A,B}(\tau)|
\\ & = |A|\cdot|B| -\tau|S_{A,B}(\tau)|.
\end{align*}
\end{proof}

\section{An additive relation between superlevel sets}

Recall the definition of superlevel sets: For
$U,V\subset\reals^1$, \begin{equation}S_{U,V}(\alpha)=\set{x:(\one_U*\one_V)(x)>\alpha}.\end{equation}
Define $A-x=\set{a-x: a\in A}$ and $-B = \set{-b: b\in B}$.

\begin{lemma} \label{lemma:converttonorm}
\begin{equation} S_{U,V}(\alpha)=\big\{x: \norm{\one_{U-x}-\one_{-V}}_1<|U|+|V|-2\alpha\big\}.  \end{equation}
\end{lemma}

\begin{proof}
This is  a direct consequence of the elementary identities 
\begin{gather} (\one_A*\one_B)(x) = |(A-x)\cap (-B)|
\\ |(A-x)\bigtriangleup (-B)| + 2|(A-x)\cap(-B)| = |A|+|B|
\\ \norm{\one_{A-x}-\one_{-B}}_1 = |(A-x)\bigtriangleup (-B)|.
\end{gather}
\end{proof}

\begin{lemma} \label{lemma:superleveldifferences}
Let $U,V\subset\reals^d$ be measurable sets with finite Lebesgue measures. Let $\alpha_1,\alpha_2>0$. Then
\begin{equation}
S_{U,V}(\alpha_1)-S_{U,V}(\alpha_2) \subset S_{U,-U}(\alpha_1+\alpha_2-|V|).
\end{equation}
\end{lemma}

\begin{proof}
Because $\one_U*\one_V$ and $\one_U*\one_{-U}$ are continuous, the associated superlevel
sets are open and there is no ambiguity in asserting that an individual point belongs to such a set.
Let $x_i\in S_{U,V}(\alpha_i)$ for $i=1,2$.
By Lemma~\ref{lemma:converttonorm},
\[ \norm{\one_{U-{x_i}}-\one_{-V}}_1<|U|+|V|-2\alpha_i.  \]
By the triangle inequality,
\[ \norm{\one_{U-{x_1}}-\one_{U-{x_2}}}_1<2|U|+2|V|-2\alpha_1-2\alpha_2.  \]
This is equivalent to
\[ \norm{\one_{U-{x}}-\one_{U}}_1<2|U|+2|V|-2\alpha_1-2\alpha_2  \]
where $x=x_1-x_2$.
By Lemma~\ref{lemma:converttonorm} again, this is in turn equivalent to
$x\in S_{U,-U}(\beta)$ where 
\[|U|+|-U|-2\beta = 2|U|+2|V|-2\alpha_1-2\alpha_2,\]
that is, $\beta = \alpha_1+\alpha_2-|V|$.
\end{proof}

\section{Analysis of the case $|A|=|B|$}

We arrive at the heart of the proof of the main theorem.
As in \cite{christRS1}, we will rely on the following characterization of near equality in the Brunn-Minkowski inequality. 
\begin{theorem} \label{thm:keystone}
Let $A,B\subset\reals^1$ be nonempty Borel sets satisfying
\begin{equation} |A+B|< |A|+|B|+\min(|A|,|B|).  \end{equation}
Then \begin{equation} \diameter(A)\le |A+B|-|B|.\end{equation} \end{theorem}

This is Proposition~3.1 of \cite{christbmtwo}; it is simply a continuum analogue of a theorem of 
Fre{\u\i}man concerning finite sumsets  of $\integers$.  It can be deduced from the finite version 
by an approximation and limiting argument.  A proof is included in \S\ref{section:keystoneproof}.

A technical point is that while the sum of two Borel sets is Lebesgue measurable, the sum of two Lebesgue measurable
sets need not be so. In this paper, this theorem will be applied only to superlevel sets, which are open.
In that case the sum set is also open, so its measurability is elementary.

In this section we show that if an ordered triple $(A,B,C)$ of subsets of $\reals^1$ nearly attains equality
in the Riesz-Sobolev inequality, and if $|A|=|B|$, then under certain auxiliary
hypotheses, $C$ is nearly equal to an interval. 
\begin{lemma}\label{lemma:equalitycase}
Let $(A,B,C)$ be an $\eta$--strictly admissible 
ordered triple of Lebesgue measurable subsets of $\reals$ having positive, finite Lebesgue measures. 
Suppose that  
\begin{align} &|A|=|B| \label{unnaturalequality} 
\\ & |A|-|C| \ge  4\scriptd(A,B,C)^{1/2} \label{unnatural}
 \\ &\scriptd(A,B,C)^{1/2} < \tfrac1{24}\eta |A|.  \end{align} 
Then there exists an interval $I\subset\reals$ such that
\begin{equation} |C\bigtriangleup I| < 14\scriptd(A,B,C)^{1/2}.  \end{equation}
\end{lemma}

Both hypotheses \eqref{unnaturalequality} and \eqref{unnatural} are unnatural from the perspective
of our main theorem, and will eventually be circumvented in \S\ref{section:conclusion}.
The main step in the proof of Lemma~\ref{lemma:equalitycase} will be: 
\begin{lemma}  \label{lemma:equalitycasemainstep}
Let $(A,B,C)$ be an $\eta$--strictly admissible 
ordered triple of subsets of $\reals$ having positive, finite Lebesgue measures. 
Let $\scriptd=\scriptd(A,B,C)$.
Suppose that $|A|=|B|$, $\scriptd^{1/2}<\tfrac12 \eta |A|$, and $|C| \le |A|-4\scriptd^{1/2}$.  Then
\begin{equation} \label{eq:AandminusA} \Big|\, |S_{A,-A}(\gamma)| - 2|C|\, \Big| \le 8\scriptd^{1/2}.  \end{equation}
\end{lemma}

\begin{proof}
The Riesz-Sobolev inequality guarantees that $\scriptd\ge 0$, while there is the trivial upper bound
\begin{equation} \label{eq:trivialUB} \scriptd 
\le \langle \one_{A^\star}*\one_{B^\star},\one_{C^\star}\rangle \le \max(|A|,|B|,|C|)^2 = |A|^2.\end{equation}

Define $\beta$ by 
\[ |C| = |A|+|B|-2\beta = 2|A|-2\beta \]
so that $\beta = |A|-\tfrac 12 |C|$ and $\gamma = |A|-|C| = 2\beta-|A|$.

By the $\eta$--strict admissibility hypothesis, $|C|\ge \eta |A| < 2\scriptd^{1/2}$.
On the other hand, $|C| \le |A| \le |A|+|A|-2\scriptd^{1/2}$ since $2\scriptd^{1/2}<\eta |A|\le |A|$.
Therefore the triple $(A,B,C)$ satisfies the hypothesis of 
Lemma~\ref{lemma:Staulowerbound}. We conclude that $S_{A,B}(\beta)$ satisfies
\begin{equation} \label{eq:8.10} \big|\,|S_{A,B}(\beta)|-|C|\,\big| \le 2\scriptd^{1/2}.  \end{equation}

By Lemma~\ref{lemma:november},
\begin{equation} \label{excellentlowerbounde} \scriptd(A,B,\beta)\le\scriptd.  \end{equation}
$|S_{A,B}(\beta)| \ge |C|-2\scriptd^{1/2}>0$,
while
$|S_{A,B}(\beta)| \le |C|+2\scriptd^{1/2}\le |A|+2\scriptd^{1/2}\le 2|A| = |A|+|B|$.
Corollary~\ref{cor:twoway} therefore applies. 
The quantity denoted by $(\sigma-\tau)^2$ in Corollary~\ref{cor:twoway}
is here $\tfrac14(|S_{A,B}(\beta)|-|C|)^2\le \scriptd$, so the Corollary gives
\begin{equation*}
\scriptd'(A,B,\beta) = \scriptd(A,B,\beta)+\tfrac14 (|S_{A,B}(\beta)|-|C|)^2 \le \scriptd+\scriptd = 2\scriptd.
\end{equation*}

Since $|B|=|A|$,
\begin{equation*} \scriptd'(A,B,\beta) = |A|\cdot|B|-\beta(|A|+|B|) + \beta^2-\int_\beta^\infty |S_{A,B}(t)|\,dt
= (|A|-\beta)^2 -\int_\beta^\infty |S_{A,B}(t)|\,dt.  \end{equation*}
Thus we have established near equality in the KPRGT inequality: 
\begin{equation}  \label{invokeKTlower}
\int_\beta^{\infty} |S_{A,B}(t)|\,dt \ge (|A|-\beta)^2 -2\scriptd.
\end{equation}
This can be equivalently written as
\begin{equation} \scriptd'(A,B,\beta)\le 2\scriptd.  \end{equation}

Up to this point, the analysis has been rather formal, involving manipulations of expressions involving integrals
of measures of superlevel sets but using very little about the definitions of those superlevel sets. 
We now introduce the underlying additive structure of the Riesz-Sobolev and KPRGT inequalities through the relation
\begin{equation} S_{A,B}(t)-S_{A,B}(t) \subset S_{A,-A}(2t-|B|).  \end{equation}
This holds for any $t>0$, by Lemma~\ref{lemma:superleveldifferences}. 

Each set $S_{A,B}(t)$ is open, so $S_{A,B}(t)-S_{A,B}(t)$ is open and hence Lebesgue measurable.
Consequently 
\begin{equation} |S_{A,B}(t)-S_{A,B}(t)|\ge 2|S_{A,B}(t)|  \end{equation}
by the Brunn-Minkowski inequality for $\reals^1$. Since $|B|=|A|$,
\begin{equation} |S_{A,-A}(2t-|A|)| = |S_{A,-A}(2t-|B|)| \ge |S_{A,B}(t)-S_{A,B}(t)| \ge 2|S_{A,B}(t)|.  \end{equation}
Equivalently, if $0<\alpha\le |A|$ then 
\begin{equation}\label{eq:usedBM} |S_{A,-A}(\alpha)| \ge 2|S_{A,B}(\tfrac12(\alpha+|A|))|.\end{equation}

As $\alpha$ varies over $[\gamma,\infty]$, $\tfrac12(\alpha+|A|)$ varies over $[\beta,\infty]$.
Consequently \eqref{eq:usedBM} can be integtated to obtain
\begin{align*}
\int_{\gamma}^{\infty} |S_{A,-A}(\alpha)|\,d\alpha
&\ge 2\int_{\gamma}^{\infty} |S_{A,B}(\tfrac12(\alpha+|A|))|\,d\alpha
\\ &=  4\int_{\beta}^{\infty} |S_{A,B}(t)|\,dt
\\ &\ge 4(|A|-\beta)^2 -8\scriptd
\\ & = 4\big(|A|^2-2\beta|A|+\beta^2 \big) -8\scriptd
\\ &= 4|A|^2-4(\gamma+|A|)|A|+(\gamma+|A|)^2 -8\scriptd
\\ &= (|A|-\gamma)^2 -8\scriptd.
\end{align*}
That is,
\begin{equation}
\scriptd'(A,-A,\gamma) \le 8\scriptd.
\end{equation}

If $|S_{A,-A}(\gamma)| \le 2|A|$ then define $\sigma$ by $|S_{A,-A}(\gamma)| = |A| + |A| -2\sigma$.
By Corollary~\ref{cor:twoway}, $(\sigma-\gamma)^2\le \scriptd'(A,-A,\gamma)\le 8\scriptd$. 
Inserting the definition of $\sigma$, this inequality becomes
\begin{equation} \Big|\, |S_{A,-A}(\gamma)|  -  \big(2|A|-2\gamma\big) \Big| \le 2(8\scriptd)^{1/2}.  \end{equation}
Since $\gamma $ was defined to be $|A|-|C|$, this is the desired upper bound for the quantity 
$\big|\,|S_{A,-A}(\gamma)|-2|C|\,\big|$. 

It is not possible to have $|S_{A,-A}(\gamma)| > 2|A|$.  Indeed, by Lemma~\ref{lemma:kprtgoutsiderange}, 
\begin{align*}
(|A|-\gamma)^2 -8\scriptd
\le  |A|^2 -\gamma|S_{A,-A}(\gamma)|.
\end{align*}
If $|S_{A,-A}(\gamma)|>2|A|$ it follows that $(|A|-\gamma)^2 -8\scriptd < |A|^2-2\gamma|A|$, so
$\gamma^2<8\scriptd$. Since $\gamma^2 = (|A|-|C|)^2$ by its definition, this contradicts the hypothesis that 
$|C|\le |A|-4\scriptd^{1/2}$. 
\end{proof}

\begin{remark}
This proof implicitly produces an upper bound for 
\[\int_\beta^\infty \big|\, |S_{A,B}(t)-S_{A,B}(t)|-2|S_{A,B}(t)| \,\big|\,dt.\]
Since $S_{A,B}(t)$ is empty for all $t\ge |A|=|B|$, this in concert with Chebyshev's inequality provides an upper bound 
for $(|A|-\beta)^{-1} \big|\, |S_{A,B}(t)-S_{A,B}(t)|-2|S_{A,B}(t)| \,\big|$
for most values of $t$. Therefore Theorem~\ref{thm:keystone} could be applied to
conclude that $S_{A,B}(t)$ nearly coincides with an interval, for most $t$.
We will argue slightly differently below to show this for the specific parameter $t=\beta$,
which is the value directly relevant to the proof of Theorem~\ref{mainthm}.
\end{remark}

\begin{remark}
The assumption $|A|=|B|$ is essential in this argument.
If $|A|>|B|$, then the lower bound \eqref{invokeKTlower} changes form.
Following the resulting changes in the ensuing steps, one arrives at a modified lower bound
for $\int_\gamma^\infty |S_{A,-A}(\alpha)|\,d\alpha$
in which the main term becomes
$4(|A|\cdot|B| -\beta|A|-\beta|B|+\beta^2)$. 
Whenever $|B|<|A|$, this quantity is strictly less than the desired $(|A|-\gamma)^2$,
which represents equality in the KPRGT inequality for $\one_A*\one_{-A}$.
\end{remark}

\begin{remark}
Because the aim is to attain an upper bound for $|S_{A,-A}(\gamma)|$,
it is natural to execute the reasoning above in the framework of near equality
for the KPRGT inequality, rather than for the Riesz-Sobolev inequality.
The latter, in its form \eqref{RSrestated}, is only relevant if the measure of
the superlevel set in question is known; but it is precisely this measure which we seek here to control.
 Thus the close connection between the two inequalities is an essential element of our reasoning.
\end{remark}

\begin{proof}[Proof of Lemma~\ref{lemma:equalitycase}]
Write $\scriptd=\scriptd(A,B,C)$.
Define $\beta,\gamma$ as above.
Since $S_{A,B}(\beta)-S_{A,B}(\beta)\subset S_{A,-A}(\gamma)$,
\begin{equation}|S_{A,B}(\beta)-S_{A,B}(\beta)|\le |S_{A,-A}(\gamma)|
\le  2|C| + 8\scriptd^{1/2}\end{equation}
by \eqref{eq:AandminusA}.
On the other hand, $2|S_{A,B}(\beta)| \ge 2|C|-4\scriptd^{1/2}$ by \eqref{eq:8.10}.  Thus 
\begin{equation*} |S_{A,B}(\beta)-S_{A,B}(\beta)| \le 2|S_{A,B}(\beta)| + 12\scriptd^{1/2}.  \end{equation*}

Since \[|S_{A,B}(\beta)|\ge |C|-2\scriptd^{1/2} \ge \eta |A|-2\scriptd^{1/2} \ge \tfrac12\eta\max(|A|,|B|,|C|), \]
$12\scriptd^{1/2} < |S_{A,B}(\beta)|$ and thus 
$|S_{A,B}(\beta)-S_{A,B}(\beta)| < 3|S_{A,B}(\beta)|$.
Therefore Theorem~\ref{thm:keystone} applies, and
certifies that $S_{A,B}(\beta)$ is contained in some interval $I$
which satisfies $|I| \le |S_{A,B}(\beta)|+12\scriptd^{1/2}$.

It has already been noted that $|C\bigtriangleup S_{A,B}(\beta)| <2\scriptd^{1/2}$.  Therefore 
\[|C\bigtriangleup I| <14\scriptd^{1/2}.\]
\end{proof}

\section{Truncations} \label{section:truncation}

Next we review and generalize a device used by Burchard \cite{burchard} and related to work of F.~Riesz \cite{riesz}, 
which makes it possible to modify the measures of the sets $A,B,C$ without sacrificing the hypothesis of near equality
in the Riesz-Sobolev inequality. 
This device will be used to remove the undesirable hypotheses \eqref{unnaturalequality}
and \eqref{unnatural} from Lemma~\ref{lemma:equalitycase}. As developed here, this device involves two free parameters
$\eta,\eta'$. 
The works \cite{burchard} and later \cite{christflock} exploited only the restricted case $\eta=\eta'$, 
but the generalization to distinct parameters $\eta,\eta'$ will be quite useful in \S\ref{section:conclusion}.

\begin{definition}
Let $S\subset\reals$ be a Lebesgue measurable set with finite, positive measure.
Let $\eta,\eta'>0$, and assume that $\eta+\eta'<|S|$.
The truncation $S_{\eta,\eta'}$  of $S$
is 
\begin{equation} S_{\eta,\eta'} = S\cap[a,b]\end{equation} where $a<b\in\reals$ are respectively the 
smallest and largest numbers satisfying
\begin{equation}\label{truncationdefn} |S\cap(-\infty,a)|=\eta \ \text{ and } \   |S\cap(b,\infty)| = \eta'.  \end{equation}
\end{definition}

\begin{lemma} \label{lemma:trunc1}
For any sets $A,B,C$ and any $\eta,\eta'\ge 0$ such that $\eta+\eta'<\min(|A|,|B|)$,
\begin{equation} \langle \one_A*\one_B,\one_C\rangle
\le \langle \one_{A_{\eta,\eta'}}*\one_{B_{\eta',\eta}},\one_C\rangle + (\eta+\eta')|C|.
\end{equation}
\end{lemma}

A point to note is that $B_{\eta',\eta}$ appears, rather than $B_{\eta,\eta'}$.
In the special case $\eta=\eta'$,
this lemma appears in the paper of Burchard \cite{burchard}, and appears to be rooted in work of  F.~Riesz \cite{riesz}. 

\begin{proof}
Consider any $x\in \reals$  and set $\tilde B = x-B$.  Then 
\[ x-B_{\eta',\eta} = (x-B)_{\eta,\eta'} = \tilde B_{\eta,\eta'}\] 
(note that $\tilde B_{\eta,\eta'}$ means $(\tilde B)_{\eta,\eta'}$
and that $\tilde B_{\eta,\eta'}$ appears, rather than $\tilde B_{\eta',\eta}$) and
\[\big(\one_{A_{\eta,\eta'}}*\one_{B_{\eta',\eta}}\big)(x) 
= \big|A_{\eta,\eta'}\cap \big(x+(-B)_{\eta,\eta'}\big)\big|
= |A_{\eta,\eta'}\cap \tilde B_{\eta,\eta'}|,\]
while
\[ \big(\one_A*\one_B\big)(x) = |A\cap (x-B)| = |A\cap \tilde B|.  \]

Let $a<a'$ and $b<b'\in\reals$ satisfy
\begin{align*}
&A_{\eta,\eta'} = A\cap [a,a'] 
\\
&\tilde B_{\eta,\eta'} =B\cap [b,b']
\end{align*}
with $a,b$ minimal and $a',b'$ maximal.
There are four possible cases to be analyzed, depending on which of $a,b$ is
larger, and which of $a',b'$ is larger. 
If for instance $a\le b$ and $a'\le b'$ then 
\begin{align*}
A\cap \tilde B 
& = \big[A_{\eta,\eta'}\cap \tilde B_{\eta,\eta'}\big]
\ \cup \big[(A\cap \tilde B\cap(-\infty,b)\big]
\ \cup \big[(A\cap \tilde B\cap(a',\infty)\big]
\\
&\subset
\big[ A_{\eta,\eta'}\cap \tilde B_{\eta,\eta'} \big]
\ \cup [\tilde B\cap(-\infty,b)]
\ \cup [A\cap(a',\infty)].
\end{align*}
Since $| \tilde B\cap(-\infty,b)|\le\eta$ and $|A\cap(a',\infty)|\le\eta'$,
\begin{equation} \label{ineq:trunc1}
|A\cap \tilde B|\le \big| A_{\eta,\eta'}\cap \tilde B_{\eta,\eta'} \big| +\eta+\eta'.
\end{equation}
The other three cases are analyzed in the same way, with the same result \eqref{ineq:trunc1}.

Thus we have shown that for every $x\in\reals$,
\[ \big(\one_{A}*\one_{B} \big)(x) \le 
\big(\one_{A_{\eta,\eta'}}*\one_{B_{\eta',\eta}} \big)(x) + \eta+\eta' \]
Integrate both sides with respect to $x\in C$ to conclude the proof.
\end{proof}

\begin{lemma} \label{lemma:trunc2}
Let $\eta,\eta'>0$.
For any intervals $I,J,K\subset\reals$ centered at $0$
and satisfying $|I|>\eta+\eta'$, $|J|>\eta+\eta'$, and $|K|\le |I|+|J|$, 
\begin{equation} \langle \one_{I}*\one_{J},\one_{K}\rangle 
= \langle \one_{(I_{\eta,\eta'})^\star}*\one_{(J_{\eta',\eta})^\star},
\,\one_K\rangle + (\eta+\eta')|K|.  \end{equation}
\end{lemma}

The verification is a straightforward calculation. This statement also appears in \cite{burchard}, in the case $\eta=\eta'$.

\begin{corollary} \label{cor:truncation}
Let $\eta,\eta',\scriptd\ge 0$. Let $A,B,C$ satisfy $|A|>\eta+\eta'$, $|B|>\eta+\eta'$, and $|C|\le |A|+|B|$. 
If 
\[ \langle \one_{A}*\one_{B},\one_{C} \rangle
= \langle \one_{A^\star}*\one_{B^\star},\one_{C^\star}\rangle - \scriptd\]
then
\[ \langle \one_{A_{\eta,\eta'}}*\one_{B_{\eta',\eta}},\one_{C} \rangle
\ge \langle \one_{(A_{\eta,\eta'})^\star}*\one_{(B_{\eta',\eta})^\star},\one_{C^\star} \rangle-\scriptd.  \]
\end{corollary}

\begin{proof}
By Lemmas~\ref{lemma:trunc1} and \ref{lemma:trunc2},
\begin{align*} 
\langle \one_{A_{\eta,\eta'}}*\one_{B_{\eta',\eta}},\one_{C} \rangle
&\ge \langle \one_A*\one_B,\one_C\rangle - (\eta+\eta') |C|
\\
&= \langle \one_{A^\star}*\one_{B^\star},\one_{C^\star}\rangle - \scriptd- (\eta+\eta') |C|
\\
&= \langle \one_{(A_{\eta,\eta'})^\star}*\one_{(B_{\eta',\eta})^\star},
\,\one_{C^\star}\rangle +(\eta+\eta') |C| - \scriptd- (\eta+\eta') |C|.
\end{align*}
\end{proof}

We pause to indicate how the Riesz-Sobolev inequality can be proved using truncations.
Let measurable sets $A,B,C\subset\reals^1$ with positive, finite Lebesgue measures be given.
If $(A,B,C)$ is not strictly admissible then we may suppose without loss of generality that $|C| \ge |A|+|B|$. Then
$\langle \one_A*\one_B,\one_C\rangle \le \norm{\one_A*\one_B}_{L^1} = |A|\cdot|B|$
while
$\langle \one_{A^\star}*\one_{B^\star},\one_{C^\star}\rangle = |A|\cdot |B|$
since $A^\star+B^\star\subset C^\star$.
If $(A,B,C)$ is strictly admissible, choose $\rho^*>0$ so that $(|A|-\rho^*)+(|B|-\rho^*)=|C|$.
Then $\rho^*<\min(|A|,|B|)$; for instance, the inequality $\rho^*<|A|$ is equivalent by a bit of
algebra to $|B|<|A|+|C|$, which holds by strict admissibility. Set $\rho=\tfrac12\rho^*$.
Then
\begin{align*}
\langle \one_{A}*\one_{B},\,\one_{C}\rangle
&\le \langle \one_{A_{\rho,\rho}}*\one_{B_{\rho,\rho}},\,\one_{C}\rangle + \rho|C|
\\& \le |A_{\rho,\rho}|\cdot|B_{\rho,\rho}| + \rho|C|
\\& = (|A|-\rho^*)(|B|-\rho^*) + \rho|C|
\\& = (|A^\star|-\rho^*)(|B^\star|-\rho^*) + \rho|C^\star|
\\& = \langle \one_{A^\star_{\rho,\rho}}*\one_{B^\star_{\rho,\rho}},\,\one_{C^\star}\rangle + \rho|C^\star|
\\& = \langle \one_{A^\star}*\one_{B^\star},\,\one_{C^\star}\rangle.
\end{align*}
\qed

\begin{lemma} \label{lemma:manytruncations}
There exists an absolute constant $K<\infty$ with the following property.
Let $A\subset\reals$ be a Lebesgue measurable set of positive, finite measure.
Let $\eps>0$ and $0<\lambda<1$.
Suppose that for any $\rho,\rho'\ge 0$ satisfying $\rho+\rho'=(1-\lambda)|A|$,
there exists an interval $I\subset\reals$ such that
\begin{equation} |A_{\rho,\rho'}\bigtriangleup I|\le \eps|A|.\end{equation}
Then there exists an interval $\scripti\subset\reals$ such that
\begin{equation}|A\bigtriangleup \scripti|\le K\lambda^{-1} \eps|A|.\end{equation}
\end{lemma}

If $\lambda$ is small then $|A_{\rho,\rho'}| = \lambda|A|$ is small relative to $A$,
so the hypothesis becomes effectively weaker, leading to the lost power $\eta^{-1}$ in the bound. 

\begin{proof}
Without loss of generality we may suppose that $\lambda = N^{-1}$ for some $N\in\naturals$,
and then that $\eps<(4N)^{-1}$.

For $j\in\set{0,1,2,\cdots,2N}$
set $A_j = A_{\rho,\rho'}$ where $\rho = \frac{j}{2N}|A|$ and 
\[\rho' = (1-N^{-1})|A|-\rho =  \frac{2N-2-j}{2N}|A|.\]
Then $|A_j| = N^{-1}|A|$ for each index $j$, and $|A_j\cap A_{j+1}| = (2N)^{-1}|A|$ for $0\le j<2N$.
There exists an interval $I_j$ satisfying $|A_j\bigtriangleup I_j|\le \eps |A|$.

For any index $j$,
\[ \big|\,|I_j|-N^{-1}|A|\,\big| 
= \big|\,|I_j|-|A_j|\,\big| \le |A_j\bigtriangleup I_j| \le \eps|A|, \]
so
\[|I_j|\ge (N^{-1}-\eps)|A|.\] 
On the other hand, for any $j\in\set{0,1,2,\cdots,2N-1}$,
\begin{align*} |I_j\bigtriangleup I_{j+1}|
&\le |I_j\bigtriangleup A_j| +|A_j\bigtriangleup A_{j+1}| +|A_{j+1}\bigtriangleup I_{j+1}|
\\ &\le \eps|A| + N^{-1}|A| + \eps|A|, \end{align*}
The assumption that $\eps<\tfrac14 N^{-1}$ is equivalent to
$(N^{-1}+2\eps)< 2(N^{-1}-\eps)$. Therefore $|I_j\bigtriangleup I_{j+1}|<|I_j|+|I_{j+1}|$.
This forces the two intervals $I_j,I_{j+1}$ to intersect.

Since $I_j$ are intervals and $I_j$ intersects $I_{j+1}$ for every $j<2N$, $\scripti=\cup_{j=0}^{2N} I_j$ is an  interval.
Since $A\setminus\scripti = \cup_j A_j\setminus\cup_j I_j \subset\cup_j (A_j\setminus I_j)$,
\[|A\setminus\scripti|\le \sum_j |A_j\setminus I_j|\le 2N\eps|A|.\]
In the same way,
\[|\scripti\setminus A|\le \sum_j |I_j\setminus A_j| \le 2N\eps|A|,\]
so \[|A\bigtriangleup \scripti|\le 4N\eps|A|\le K\lambda^{-1}\eps|A|.\]
\end{proof}

\section{The general case} \label{section:conclusion}
We now use truncations to complete the proof of Theorem~\ref{mainthm}. 
Let $(A,B,C)$ be an $\eta$--strictly admissible 
ordered triple of Lebesgue measurable subsets of $\reals^1$ with finite, positive measures.

The analysis is broken into cases, depending on the relative sizes of $|A|,|B|,|C|$.
\begin{lemma} \label{generalcasefirstlemma}
If
\begin{equation}\label{gc1lemmahypothesis} \scriptd(A,B,C)^{1/2} \le \tfrac1{48}\eta^2
\max(|A|,|B|,|C|)\end{equation}
and
\begin{equation} |A|>\max(|B|,|C|)- \tfrac14 \eta\max(|A|,|B|,|C|) \end{equation}
then there exists an interval $\scripti$ satisfying
\begin{equation} |A\bigtriangleup \scripti|
\le K\eta^{-1}\scriptd(A,B,C)^{1/2}.\end{equation}
\end{lemma}

\begin{proof}
Let $\scriptd=\scriptd(A,B,C)$. 
Choose $\delta>0$ so that
\begin{equation} 4\scriptd^{1/2}\le\delta\le\tfrac1{8}\eta\max(|A|,|B|,|C|).  \end{equation}
Define the nonnegative quantities 
\begin{align*}
\sigma^* &= |A|-|B|+\delta 
\\ \rho^* &= |A|-|C|+\delta. 
\end{align*}
For any nonnegative parameters $\rho,\rho',\sigma,\sigma'$ satisfying
$\rho+\rho'=\rho^*$ and $\sigma+\sigma'=\sigma^*$,
define 
\begin{equation*}
\scripta = A_{\rho+\sigma,\rho'+\sigma'},
\qquad \scriptb = B_{\rho',\rho},
\qquad \scriptc = C_{\sigma,\sigma'}.
\end{equation*}


We next verify that $(\scripta,\scriptb,\scriptc)$ satisfies the hypotheses of Lemma~\ref{lemma:equalitycase}.
\begin{align*} |\scripta| &=|A|-\rho^*-\sigma^* 
\\ &= |B|+|C|-|A|-2\delta
\\ &\ge \eta \max(|A|,|B|,|C|) -2\delta 
\\ &\ge \tfrac12 \eta\max(|A|,|B|,|C|)  \end{align*}
since  $\delta\le \tfrac14 \eta\max(|A|,|B|,|C|)$.
Therefore $\scripta,\scriptb,\scriptc$ all have positive Lebesgue measures,
$|\scriptb|=|\scriptc| = |B|+|C|-|A|-\delta$, and $|\scripta| = |\scriptb|-\delta$.
Moreover, 
\[ |\scripta|+|\scriptb|-|\scriptc| = |\scripta|\ge \tfrac12\eta \max(|A|,|B|,|C|)
\ge\tfrac12\eta\max(|\scripta|,|\scriptb|,|\scriptc|), \]
so the triple $(\scripta,\scriptb,\scriptc)$ is $\tfrac12\eta$--strictly admissible.
By two consecutive applications of Corollary~\ref{cor:truncation},
\[ \langle \one_{\scripta}*\one_{\scriptb},\,\one_{\scriptc}\rangle
\ge \langle \one_{\scripta^*}*\one_{\scriptb^*},\,\one_{\scriptc^*}\rangle -\scriptd. \]
Since $|\scriptc|-|\scripta| =\delta \ge 4\scriptd^{1/2}$, $(\scriptb,\scriptc,\scripta)$
satisfies \eqref{unnaturalequality}.
Finally \eqref{unnatural} holds with $\eta$ replaced by $\tfrac12\eta$ because
\begin{multline*} |\scriptb| = |B|+|C|-|A|-\delta \ge\eta\max(|A|,|B|,|C|)-\delta 
\ge \tfrac12 \eta \max(|A|,|B|,|C|), \end{multline*} so that $\scriptd^{1/2} \le \tfrac1{24}\cdot\tfrac12 \eta |\scriptb|$. 
Thus $(\scriptb,\scriptc,\scripta)$ satisfies the hypotheses of Lemma~\ref{lemma:equalitycase}, with $\eta$ replaced by $\tfrac12\eta$.
We conclude that there exists an interval $I$ such that
\begin{equation} |\scripta\bigtriangleup I| \le 14\scriptd^{1/2}.\end{equation}

This has been proved for $\scripta = A_{\rho+\sigma,\rho'+\sigma'}$
whenever $\rho+\rho'=\rho^*$ and $\sigma+\sigma'=\sigma^*$. 
By Lemma~\ref{lemma:manytruncations}, this implies that
there exists an interval $\scripti$ satisfying
$|A\bigtriangleup \scripti| \le K\lambda^{-1}\scriptd^{1/2}$
where $K<\infty$ is an absolute constant and
\begin{align*} \lambda &= 1-\frac{\rho^*+\sigma^*}{|A|}
\\ &= \frac{|B|+|C|-|A|-2\delta}{|A|}
\\ & \ge \eta -2\delta |A|^{-1}
\\ &\ge \eta - \tfrac28 \eta |A|^{-1}\max(|A|,|B|,|C|).
\end{align*}
The hypothesis \eqref{gc1lemmahypothesis} guarantees that $|A|\ge\tfrac12 \max(|A|,|B|,|C|)$ and therefore $\lambda\ge\tfrac12\eta$.
Thus $|A\bigtriangleup \scripti| \le K\eta^{-1}\scriptd^{1/2}$.
\end{proof}

Next consider an arbitrary $\eta$--strictly admissible triple $(A,B,C)$.  Continue to simplify notation by writing
$\scriptd=\scriptd(A,B,C)$. By permuting these sets and invoking the identities
\begin{align*}
&\langle \one_A*\one_B,\one_C\rangle  
=\langle \one_A*\one_{-C},\one_{-B}\rangle  
\\
&\langle \one_A*\one_B,\one_C\rangle  
=\langle \one_B*\one_A,\one_C\rangle, 
\end{align*}
we may suppose that $|A|\ge |B|\ge |C|$.
Then $|A|\ge\max(|B|,|C|)\ge\max(|B|,|C|)-\tfrac14\eta\max(|A|,|B|,|C|)$, 
so Lemma~\ref{generalcasefirstlemma} can be applied to conclude that $A$ has suitably small symmetric difference with some interval.
If $|B|\ge |A|-\delta$, then the same applies also to $B$; likewise for $C$ if $|C|\ge |A|-\delta$.
\qed

\medskip

Continuing to assume that $|A|\ge |B| \ge |C|$, consider the case in which $|B|<|A|-\tfrac14\eta |A|$, 
so that Lemma~\ref{generalcasefirstlemma} can no longer be applied directly to $B$.
Consider the triple $(\tilde A,\tilde B,\tilde C)$ with $\tilde B =A_{\rho,\rho'}$, $\tilde A=B$,
and $\tilde C= C_{\rho,\rho'}$ where $\rho,\rho'$ are nonnegative
and $\rho+\rho'=\rho^* = |A|-|B|$. Then $|\tilde A|=|\tilde B|>|\tilde C|$, and 
\[|\tilde A|-|\tilde C| = |B|-(|C| - (|A|-|B|)) = |A|-|C| \ge \tfrac14\eta |A|^2 
\ge\tfrac14\eta\max(|\tilde A|,|\tilde B|,|\tilde C|).\]
Moreover \[|\tilde B|+|\tilde C|-|\tilde A| = |\tilde C| = |C|-\rho^* = |C|+|B|-|A|\ge\eta|A| \ge \eta|\tilde A|,\]
so $(\tilde A,\tilde B,\tilde C)$ is $\eta$--strictly admissible. 
And
\begin{equation*} \max(|\tilde A|,|\tilde B|,|\tilde C|) = |\tilde A| = |B| \ge \eta\max(|A|,|B|,|C|); \end{equation*}
there is a loss of a factor of $\eta$ here in comparison with $\max(|A|,|B|,|C|)$.

The discrepancy 
$\tilde\scriptd = \langle \one_{\tilde A^\star}*\one_{\tilde B^\star},\,\one_{\tilde C^\star}\rangle
- \langle \one_{\tilde A}*\one_{\tilde B},\,\one_{\tilde C}\rangle$
satisfies $\tilde\scriptd\le\scriptd$. 
Assuming that $\scriptd^{1/2}\le K^{-1}\eta^4|A|$ for a sufficiently large constant $K$,
it follows that \[\tilde\scriptd^{1/2}\le K^{-1}\eta^2 |\tilde A|.\]
Since $|\tilde A|-|\tilde C|\ge \tfrac14\eta\max(|\tilde A|,|\tilde B|,|\tilde C|)$, 
Lemma~\ref{generalcasefirstlemma} can be applied to $(\tilde A,\tilde B,\tilde C)$ to
conclude that there exists an interval $\scriptj$ satisfying
\begin{equation*} |\tilde A\bigtriangleup \scriptj| \le K\eta^{-1}\scriptd^{1/2}.\end{equation*}
$\tilde A$ was defined to equal $B$, so this is the desired conclusion for $B$.
\qed
\medskip

We have shown  thus far that in all cases, both $A,B$ nearly coincide with intervals,
in the desired sense. Moreover, the same conclusion holds for $C$ if $|C|\ge |A|-\tfrac14\eta |A|$.
It remains only to analyze $C$, under the assumption that $|C|<|A|-\tfrac14 \eta |A|$.
Let $\rho^* = |A|-|B|$.
Consider $(\tilde A,\tilde B,\tilde C) = (A_{\rho,\rho'},\,B,\,C_{\rho,\rho'})$ 
where $\rho,\rho'$ are nonnegative and $\rho+\rho' = \rho^*$.
Then $|B|=|\tilde A|=|\tilde B|>|\tilde C|$, and 
\[|\tilde C| = |C|-(|A|-|B|) = |B|+|C|-|A| \ge \eta|A|\ge\eta|\tilde A|.\]
Moreover \[|\tilde A|-|\tilde C| = |A|-|C| \ge \tfrac14\eta|A| \ge \tfrac14\eta|\tilde A|.\]
Therefore if $\scriptd^{1/2}\le K^{-1}\eta^4\max(|A|,|B|,|C|)$ for sufficiently large $K$
then Lemma~\ref{generalcasefirstlemma} can be applied once more to give the desired conclusion for $C$.
This concludes the proof of Theorem~\ref{mainthm}.
\qed

\section{Centers of intervals}
Let $A,B,C,\eta,\eps$ be as in the statement of Theorem~\ref{mainthm}.
The theorem states that there exist intervals $I,J,L$
that satisfy its first conclusion \eqref{eq:conclusion1}, with the additional property \eqref{eq:conclusion2} that their
centers $a,b,c$ satisfy $|a+b-c|\le K\eta^{-2}\eps^{1/4}\max(|A|,|B|,|C|)$.
We will show that \eqref{eq:conclusion2}
holds for the centers of any intervals $I,J,L$ that satisfy \eqref{eq:conclusion1}, that is,
$|A\bigtriangleup I|\le K\eta^{-1}\eps^{1/2} \max(|A|,|B|,|C|)$
and likewise for $J,L$ respectively relative to $B,C$. Indeed,
denoting by $K$ a constant whose value is allowed to change from one occurrence to the next,
\begin{align*}
\langle \one_I*\one_J,\,\one_L\rangle
&\ge
\langle \one_{A\cap I}*\one_{B\cap J},\,\one_{C\cap L}\rangle
\\& \ge
\langle \one_{A}*\one_{B},\,\one_{C}\rangle
- K\max(|A|,|B|,|C|)\max(|A\setminus I|,|B\setminus J|,|C\setminus L|)
\\ &\ge
\langle \one_{A}*\one_{B},\,\one_{C}\rangle
- K\eta^{-1} \eps^{1/2}\max(|A|,|B|,|C|)^2
\\ &\ge
\langle \one_{A^\star}*\one_{B^\star},\,\one_{C^\star}\rangle
- K\eta^{-1} \eps^{1/2}\max(|A|,|B|,|C|)^2 - \eps \max(|A|,|B|,|C|)^2
\\ &\ge
\langle \one_{A^\star}*\one_{B^\star},\,\one_{C^\star}\rangle
- K\eta^{-1} \eps^{1/2}\max(|A|,|B|,|C|)^2 
\\ &\ge
\langle \one_{I^\star}*\one_{J^\star},\,\one_{L^\star}\rangle
- K\eta^{-1} \eps^{1/2}\max(|A|,|B|,|C|)^2 
\\& \qquad\qquad
- K\max(|A|,|B|,|C|) \max\big( \big|\,|A|-|I|\,\big|,
\big|\,|B|-|J|\,\big|,
\big|\,|C|-|L|\,\big|, \big)
\\ &\ge
\langle \one_{I^\star}*\one_{J^\star},\,\one_{L^\star}\rangle
- K\eta^{-1} \eps^{1/2}\max(|A|,|B|,|C|)^2 .
\end{align*}

The ordered triple $(I,J,L)$ satisfies
\begin{align*}
|I|+|J|-|L| &\ge |A|+|B|-|C| 
-K\eps^{1/2}\eta^{-1}
\max(|A|,|B|,|C|) 
\\ &\ge \eta \max(|A|,|B|,|C|)  -K\eps^{1/2}\eta^{-1}
\\ &\ge \eta \max(|K|,|J|,|L|)  -K\eps^{1/2}\eta^{-1}
\max(|A|,|B|,|C|) 
\\ &\qquad
-\eta K\eps^{1/2}\eta^{-1} \max(|A|,|B|,|C|) 
\\ &\ge \eta \max(|K|,|J|,|L|)  -K\eps^{1/2}\eta^{-1}
\max(|A|,|B|,|C|)
\\ &\ge \tfrac12 \eta \max(|K|,|J|,|L|) 
\end{align*}
provided that $\eps$ is sufficiently small, since it is assumed that $\eta = O(\eta^2)$.
Therefore $(I,J,L)$ is $\tfrac12 \eta$--strictly admissible.

Thus we have reduced matters to the situation in which $(A,B,C)$ is replaced by $(I,J,L)$,
$\eta$ is replaced by $\tfrac12\eta$, and $\eps$ is replaced by $K\eta^{-1}\eps^{1/2}$.
In this situation it is elementary that $|a+b-c|\le K\eta^{-1} \eps^{1/4}$.

\section{Near equality in the KPRGT inequality}
The following is an analogue for the KPRGT inequality of our main theorem.
\begin{theorem} \label{thm:KT}
For any $\eta\in(0,\tfrac12]$
there exists an absolute constant $K<\infty$ for which the following holds.
Let $A,B$ be 
measurable subsets of $\reals^1$ with finite, positive Lebesgue measures
satisfying \[\min(|A|,|B|)\ge\eta(1-\eta)^{-1}\max(|A|,|B|).\]
Let \begin{equation}\label{KTtauhypothesis}
\tau\in[\,\eta\max(|A|,|B|),\,\,(1-\eta)\min(|A|,|B|)\,].\end{equation}
If the deficit
\begin{equation} \label{eq:KT:scriptddefn}
\scriptd'=\scriptd'(A,B,\tau) = \Big[|A|\cdot|B|-\tau(|A|+|B|) +\tau^2\Big] - \int_\tau^\infty |S_{A,B}(t)|\,dt  
\end{equation}
satisfies
\begin{equation} \scriptd' <\min(\tau^2,(|B|-\tau)^2,K^{-1}\eta^8\max(|A|,|B|)^2) \end{equation}
then there exist intervals $I,J$ satisfying
$|A\bigtriangleup I|\le K\sqrt{\scriptd'}$ and $|B\bigtriangleup J|\le K\sqrt{\scriptd'}$.
\end{theorem}

We will deduce this from Theorem~\ref{mainthm}. A preliminary step is to control the measure of $S_{A,B}(\tau)$.
\begin{lemma} \label{lemma:kprgtthmprelim} Under the hypotheses of Theorem~\ref{thm:KT},
\begin{equation}  \big|\,|A|-|B|\,\big|  \le |S_{A,B}(\tau)| \le |A|+|B|.  \end{equation} \end{lemma}

\begin{proof}
Write $S = S_{A,B}(\tau)$.
Assume without loss of generality that $|A|\ge |B|$.
If $|S|< |A|-|B|$ then
\[ \int_\tau^\infty |S_{A,B}(t)|\,dt 
= \langle \one_A*\one_B,\,\one_{S}\rangle - \tau|S| \le |B|\cdot|S|-\tau|S|.  \]
Now
\[ |B|\cdot|S|-\tau|S| -|A|\cdot|B|+\tau(|A|+|B|)-\tau^2 = (|B|-\tau)(|S|-(|A|-|B|)-(|B|-\tau)^2. \]
Therefore 
\[ \scriptd'\ge (|B|-\tau)(|A|-|B|-|S|) +(|B|-\tau)^2. \]
Since $|B|>(1-\eta)|B|=(1-\eta)\min(|A|,|B|)\ge\tau$
and $|A|>|B|+|S|$, this contradicts the hypothesis that $\scriptd'< (|B|-\tau)^2$.

If on the other hand $|S|>|A|+|B|$ then 
\[ \int_\tau^\infty |S_{A,B}(t)|\,dt \le |A|\cdot|B| - \tau|S| \]
so
\[ \scriptd' \ge |A|\cdot|B|-\tau(|A|+|B|)+\tau^2-\big(|A|\cdot|B|-\tau|S|\big)
= \tau(|S|-|A|-|B|)+\tau^2,  \]
contradicting the assumption that $\scriptd'<\tau^2$.
\end{proof}

\begin{proof}[Proof of Theorem~\ref{thm:KT}]
Define $\sigma$ by $|S_{A,B}(\tau)| = |A|+|B|-2\sigma$.
Lemma~\ref{lemma:kprgtthmprelim} asserts that $(A,B,S_{A,B}(\tau))$ satisfies the hypothesis of
Corollary~\ref{cor:twoway}. Therefore
\[\langle \one_{A^*}*\one_{B^*},\one_{S_{A,B}(\tau)^*}\rangle -\langle \one_{A}*\one_{B},\one_{S_{A,B}(\tau)}\rangle \le\scriptd'\]
and 
\[ \Big|\,|S_{A,B}(\tau)|-\big(|A|+|B|-2\tau \big)\,\Big| = 2|\sigma-\tau| \le 2\scriptd'^{1/2}.  \]

Simple calculations using the hypotheses  \eqref{KTtauhypothesis}
and $\scriptd'\le \tfrac14\eta^2\max(|A|,|B|)^2$ show that
\[|S_{A,B}(\tau)|\le (1+2\eta)(|A|+|B|)\le (2+4\eta)\max(|A|,|B|)\]
and that the ordered triple 
$(A,B,S_{A,B}(\tau))$ is strictly $\gamma$--admissible, where 
\[\gamma = 2\eta^2-2\scriptd'^{1/2}\max(|A|,|B|)^{-1}\ge\eta^2.  \]

Theorem~\ref{mainthm} applies and yields the desired conclusion, provided that
\[\max(|A|,|B |)^{-2} \scriptd' \le K^{-1}\gamma^4,\] that is, $\scriptd'\le K^{-1}\gamma^4\max(|A|,|B|)^2$.
The hypothesis $\scriptd' \le K^{-1}\eta^8 \max(|A|,|B|)^2$ ensures this.
\end{proof}

\section{Equivalence of two formulations} \label{section:equivalence}

As was shown in Lemma~\ref{RSrestated}, the Riesz-Sobolev inequality implies the upper bound 
\[ \int_{S_{A,B}(\tau)} \one_A*\one_B \ \le\  |A|\cdot|B|-\sigma^2 \]
where $|S_{A,B}(\tau)| = |A|+|B|-2\sigma$,
provided that $|S_{A,B}(\tau)|$ lies in $[\max(|A|,|B|)-\min(|A|,|B|),\,|A|+|B|]$.
The Riesz-Sobolev inequality concerns general sets, rather than only superlevel sets, but 
is nonetheless a consequence of this one by formal reasoning.
For the sake of completeness we indicate here a proof. 
In conjunction with Corollary~\ref{cor:twoway}, 
this provides a proof that the KPRGT inequality for $\reals^1$ implies the Riesz-Sobolev inequality for $\reals^1$.

Recall the notation $\sigma(t)=\sigma$ where $t=|A|+|B|-2\sigma$.
It suffices to consider $\int_E \one_A*\one_B$ when $0<|E|<|A|+|B|$,
for sets $E\subset\set{x: (\one_A*\one_B)(x)>0}$.
Among all sets $E$ of a given measure $|E|$, $\int_E \one_A*\one_B$ is maximized 
when $E$ is a superlevel set $S_{A,B}(\tau)$, where $\tau$ is chosen so that
$|S_{A,B}(\tau)|=|E|$, provided that such a value of $\tau$ exists.
In this case, \eqref{RSrestated} for this value of $\tau$ implies that
$\int_E \one_A*\one_B \le |A|\cdot|B|-\sigma(|E|)^2$.

Consider next the case in which $0<|E|<|A|+|B|$, but no such parameter $\tau$ exists. 
By the Brunn-Minkowski inequality, $\set{x: (\one_A*\one_B)(x)>0}$
has measure $\ge |A|+|B|>|E|$. $\tau\mapsto |S_{A,B}(\tau)|$
is a nonincreasing upper semicontinuous function. Therefore
there exists $\tau$ such that $|S_{A,B}(\tau)|<|E|$,
but $\scripts=\set{x: (\one_A*\one_B)(x)=\tau}$ satisfies $|\scripts| + |S_{A,B}(\tau)|\ge |E|$.
Write $S=S_{A,B}(\tau)$.
Choose a measurable set $\scripts^\dagger\subset\set{x: (\one_A*\one_B)(x)=\tau}$
Then $\int_E \one_A*\one_B\le \int_{E^\dagger} \one_A*\one_B$
for any subset $E^\dagger$ of $S\cup\scripts$ satisfying $|E^\dagger|=|E|$ and $E^\dagger\supset S$.
So it suffices to bound $\int_{E^\dagger}\one_A*\one_B$ for such sets.

Let $T=S\cup\scripts$ if $|S|+|\scripts|\le |A|+|B|$, and if $|S|+|\scripts|>|A|+|B|$ let 
$T$ be a measurable set satisfying $S\subset T\subset S\cup\scripts$ with $|T| = |A|+|B|$.  
\begin{align*}
\int_{E^\dagger} \one_A*\one_B = \int_{T}\one_A*\one_B -\tau(|T|-|E|)
\le |A|\cdot|B| -\sigma(|T|)^2 -\tau(|T|-|E|).
\end{align*}
By the same reasoning, since $S\subset E^\dagger$ and $\one_A*\one_B\equiv \tau$ on $E^\dagger\setminus S$,
\begin{align*}
\int_{E^\dagger} \one_A*\one_B 
\le |A|\cdot|B|
-\sigma(|S|)^2 +\tau(|E|-|S|),
\end{align*}
so it suffices to verify that
\begin{equation} \label{soonconcave}
\min \Big( -\sigma(|T|)^2 -\tau|T| \, , \, -\sigma(|S|)^2 -\tau|S| \Big)
\le -\sigma(|E|)^2-\tau|E|.
\end{equation}
Set $a=|S|$, $b=|T|$, and $x=|E|$.
For $t\in\reals$ define \[h(t) = -\sigma(t)^2 -\tau t = -\tfrac14(|A|+|B|-t)^2-\tau t.\] 
In these terms, \eqref{soonconcave} becomes
\begin{equation} \label{concave} \min\big(h(a),h(b)\big)\le h(x).  \end{equation}
Since $a\le x\le b$, and since $h$ is concave, \eqref{concave} holds.

\section{Proof of Theorem~\ref{thm:keystone}} \label{section:keystoneproof}

Theorem~\ref{thm:keystone}, a continuum version of a theorem of Fre{\u\i}man concerned with finite sets, 
was proved in \cite{christRS1}. We repeat the proof here, again for the sake of completeness.
\begin{lemma} \label{lemma:localizenoloss} Let $A,B\subset\reals$ be compact sets. Let $c\in\reals$ and $I=(-\infty,c]$.
If $B\cap I\ne\emptyset$ then \[|A+(B\cap I)|-|A|-|B\cap I| \le |A+B|-|A|-|B|.\] \end{lemma}

\begin{proof}
Since $B$ is closed, we may suppose without loss of generality that $c\in B$.
By independently translating $A,B$ we may assume that $c=0$ and that the maximal element of $A$ is equal to $0$.
Then $A+(B\cap I)\subset(-\infty,0]$, while $B\setminus I = \set{0}+(B\setminus I)\subset (0,\infty)$.
Since $A+B\supset\set{0}+(B\setminus I)$, $A+B$ contains both $A+(B\cap I)$ and $\set{0}+(B\setminus I)$,
which have at most the single point $0$ in common. Therefore \[ |A+B| \ge |A+(B\cap I)| + |B\setminus I|, \]
and consequently \[ |A+(B\cap I)|-|A|-|B\cap I| \le |A+B|-|B\setminus I| -|A|-|B\cap I| = |A+B|-|A|-|B|.  \]
\end{proof}

It suffices to prove Theorem~\ref{thm:keystone} for compact sets.
Indeed, let $A,B$ be measurable sets satisfying the hypothesis. 
It suffices to prove that any compact subset $K\subset A$ has diameter $\le |A+B|-|B|$.
Let $\eps>0$ be arbitrary.
Choose compact sets $\tilde K$ satisfying $K\subset\tilde K\subset A$
and $L\subset B$ such that $|A|-|\tilde K|$ and $|B|-|L|$ are sufficiently small to ensure that
$|\tilde K|+|L|+\min(|\tilde K|,|L|) > |A+B|$
and $|A+B|-|L|<|A+B|-|B|+\eps$.
Then $|\tilde K+L|-|L|\le |A+B|-|L|<|A+B|-|B|+\eps$ and
\[|\tilde K+L|<|\tilde K|+|L|+\min(|\tilde K|,|L|).\]
Therefore if the conclusion of Theorem~\ref{thm:keystone}
is assumed to hold for compact sets, it follows that for any $\eps>0$, 
\[\diameter(K)\le \diameter(\tilde K)<|A+B|-|B|+\eps.\]

\begin{proof}[Proof of Theorem~\ref{thm:keystone}]
Let $A,B$ be arbitrary compact sets that satisfy $|A+B|<|A|+|B|+\min(|A|,|B|)$. 
Consider first the case in which $\diameter(A)\ge\diameter(B)$.
By scaling we may assume without loss of generality that $\diameter(A)=1$.
By independently translating $A,B$ we may assume that $0=\min(A)$ and $1=\max(A)$.
Thus $A+B\subset[0,2]$.

Let $\pi:\reals\to\torus=\reals/\integers$ be the quotient map.
Since $A,B$ have diameters $\le 1$, $|\pi(A)|=|A|$ and $|\pi(B)|=|B|$.
Since $0\in A$, $B\subset A+B$ and therefore $\pi(B)\subset\pi(A+B)$.

We claim that
\begin{equation} \label{preKclaim}
|A+B| \ge |\pi(A+B)| + |B|.
\end{equation}
Indeed, there exists at least one measurable set $S\subset A+B$ satisfying both 
\begin{equation*}
\left\{\begin{aligned}
\pi(S)&=\pi(A+B)\setminus\pi(B)
\\ |S| &= |\pi(A+B)\setminus\pi(B)|. 
\end{aligned} \right.
\end{equation*}
Choose such a set $S$.

Then $A+B$ contains the three sets $B$, $\set{1}+B$, and $S$.
Since $\pi(B)=\pi(\set{1}+B)$, and since $\pi(S)$ is disjoint from $\pi(B)$,
$S$ is disjoint from $B\cup(\set{1}+B)$. The two sets $B\subset[0,1]$ and $\set{1}+B\subset [1,2]$ 
can have at most one element in common since $\diameter(B)\le 1$. Therefore
\begin{align*} |A+B|&\ge |B|+|\set{1}+B| + |S| 
\\& = 2|B|+|S|
\\& = 2|\pi(B)|+|\pi(A+B)\setminus \pi(B)|
\\& = 2|\pi(B)| + |\pi(A+B)|-|\pi(B)| 
\\ & = |\pi(A+B)|+|\pi(B)|
\\ & = |\pi(A+B)|+|B|,  \end{align*}
as claimed.

A theorem of Kemperman \cite{kemperman1964} 
(see also \cite{ruzsa} and \cite{taospending} for an alternative proof) 
states that for any Borel subsets of $\torus$,
$|\scripta+\scriptb| \ge \min(|\scripta|+|\scriptb|,1)$.
Apply this to $\pi(A),\pi(B)$, noting that $\pi(A)+\pi(B) = \pi(A+B)$,
to conclude that
\begin{equation} \label{eq:projlowerbound} |\pi(A+B)|\ge \min(|\pi(A)|+|\pi(B)|,1) = \min(|A|+|B|,1).\end{equation}

If $|A|+|B|\le 1$ then $|\pi(A+B)|\ge |A|+|B|$ and therefore by \eqref{preKclaim}, 
$|A+B|\ge |A|+|B| + |B|$, contradicting the hypothesis $|A+B|<|A|+|B|+\min(|A|,|B|)$.  Therefore $|A|+|B|\ge 1$. 
By \eqref{eq:projlowerbound}, this forces $|\pi(A+B)|\ge 1$.

By \eqref{preKclaim}, $|A+B|\ge 1+|B|$. By the normalization $1=\diameter(A)$, this is equivalent to
\[ \diameter(A)=1\le |A+B|-|B|, \]
as was to be proved.

There remains the case in which $\diameter(A)<\diameter(B)$. Applying the case already treated with the roles
of $A,B$ reversed then gives $\diameter(B)\le |A+B|-|A|$, so by transitivity $\diameter(A)<|A+B|-|A|$.
If $|A|\ge |B|$, this gives a stronger bound than required.
In particular, we have established the desired inequality whenever $|A|=|B|$,
regardless of which set has the larger diameter.

If $|B|>|A|$, choose $I=(-\infty,c]$ so that $|B\cap I|=|A|$. 
By Lemma~\ref{lemma:localizenoloss}, 
\[|A+(B\cap I)|  \le |A|+|B\cap I| + (|A+B|-|A|-|B|).\] 
Therefore
\begin{align*} |A+(B\cap I)| &< |A|+ |B\cap I| + \min(|A|,|B|) 
\\&= |A|+|B\cap I|+\min(|A|,|B\cap I|),\end{align*}
so we can conclude from what is shown above that
\begin{align*}
\diameter(A)& \le |A+(B\cap I)|-|B\cap I| 
\\ & = \big(|A+(B\cap I)| - |A| - |B\cap I|\big) +|A|
\\ & \le (|A+B|-|A|-|B|)+|A| \\ &= |A+B|-|B|.
\end{align*}
\end{proof}



\begin{thebibliography}{20}


\bibitem{burchard} A.~Burchard,
{\em Cases of equality in the Riesz rearrangement inequality}, 
Ann. of Math. (2) 143 (1996), no. 3, 499--527

\bibitem{christradon} M.~Christ, 
{\em Extremizers of a Radon transform inequality}, preprint  math.CA arXiv:1106.0719,
to appear in proceedings of Princeton symposium in honor of E.~M.~Stein 

\bibitem{christRS1} \bysame,
{\em An approximate inverse Riesz-Sobolev rearrangement inequality}, preprint,
math.CA arXiv:1112.3715

\bibitem{christyoungest} \bysame,
{\em On near-extremizers for Young's inequality for $\reals^d$}, preprint,
math.CA arXiv:1112.4875

\bibitem{christbmtwo} \bysame,
{\em Near equality in the two-dimensional Brunn-Minkowski inequality},
math.CA arXiv:1206.1965

\bibitem{christbmhigh} \bysame,
{\em Near equality in the Brunn-Minkowski inequality},
math.CA arXiv:1207.5062 

\bibitem{christflock}
M.~Christ and T.~Flock,
{\em Cases of equality in certain multilinear inequalities of Hardy-Riesz-Brascamp-Lieb-Luttinger type},
preprint, math.CA arXiv:1307.8439


\bibitem{figallijerison}
A.~Figalli and D.~Jerison,
{\em On the addition of sets in $\reals^n$: a quantitative stability result},
preprint 2013

\bibitem{greenruzsa}
B.~Green and  I.~Z.~Ruzsa, 
{\em Sum-free sets in abelian groups}, 
Israel J. Math. 147 (2005), 157--188


\bibitem{kemperman1964}
J.~H.~B.~Kemperman,
{\em On products of sets in a locally compact groups},
Fund. Math. 56 (1964), 51--68, Israel J. Math. 147 (2005), 157--188

\bibitem{liebHLS}
E.~Lieb, {\em Sharp constants in the Hardy-Littlewood-Sobolev and related inequalities}, 
Ann. of Math. (2) 118 (1983), no. 2, 349--37

\bibitem{liebloss} E.~H.~Lieb and M.~Loss, 
{\em Analysis}, Amer. Math. Soc., Providence, RI, 1997

\bibitem{pollard}
J.~M.~Pollard, 
{\em A generalisation of the theorem of Cauchy and Davenport}, 
J. London Math. Soc. (2) 8 (1974), 460--462


\bibitem{riesz}
F.~Riesz,
{\em Sur une in\'egalit\'e int\'egrale}, Journal London Math. Soc. 5 (1930), 162--168

\bibitem{ruzsa}
I.~Z.~Ruzsa,
{\em A concavity property for the measure of product sets in groups}, 
Fund. Math. 140 (1992), no. 3, 247--254


\bibitem{sobolev} S.~L.~Sobolev,
{\em On a theorem of functional analysis}, Mat. Sb. (N.S.) 4 (1938), 471--479,
A. M. S. Transl.  Ser. 2, 34 (1963), 39-68

\bibitem{taoblog}
T.~Tao,  \url{http://terrytao.wordpress.com/2011/12/26/a-variant-of-kempermans-theorem/}.

\bibitem{taospending}
\bysame,
{\em Spending Symmetry}, to appear

\bibitem{taovu} T.~Tao and V.~Vu, {\em Additive Combinatorics},
Cambridge Studies in Advanced Mathematics, 105. Cambridge University Press, Cambridge, 2006


\end{thebibliography}
\end{document}